%% file: CotorsionPairsAndQuillenAdjunctions.tex
\documentclass[twoside]{skript}

\usepackage[
	final
	]{showlabels}
\usepackage{adjustbox}
\usepackage{makebox}
\usepackage{csquotes}

\newcommand{\skipind}[1]{\overset{\widehat{#1}}{\ldots}} 
\newcommand{\empind}[1]{\ldots,\overset{#1}{-},\ldots} 
\newcommand{\cf}{cf.\@}
\DeclareMathOperator{\Ac}{Ac}

\newcommand{\multi}[1]{multi#1}
\newcommand{\multimodelcategories}{multi-model-categories}

\newcommand{\Rsite}{\textbf{RSite}}

\title{Cotorsion Pairs and Quillen Adjunctions}
\author{Ren\'{e} Recktenwald}

\begin{document}


\maketitle
\begin{abstract}
	\noindent Let $F:\Aa\rightleftarrows \Bb:G$ be an adjunction between abelian categories and $\Ch(\Aa)\rightleftarrows \Ch(\Bb)$ the induced adjunction on chain complexes. Let $\Aa$ and $\Bb$ be equipped with sufficiently nice cotorsion pairs, then we find model structures on their categories of chain complexes. This paper gives novel criteria under which the induced adjunction becomes a Quillen adjunction. In particular our criteria can be checked on the level of the underlying abelian categories instead of chain complexes. We do the same for adjunctions of $n$ variables and we also show how our results imply that the flat model structure can be used to compute the derived tensor product.
\end{abstract}
\tableofcontents
\section{Introduction}
	\input{Introduction}

\section{Setting and Notation}\label{sec:setting-and-notation}
	\input{SettingNotation}

\section[The 1-variable case]{The $1$-variable case}\label{sec:1-var-case}

\input{1VariableCase}

\section[The \textit{n}-variable case]{The $n$-variable case}\label{sec:n-var-case}
	\input{nVariableCase}

\section{Applications}\label{sec:applications}
	
	In this section we show how one can use the results of this paper to show that the flat model structure makes the usual four functor formalism into a Quillen four functor formalism. More precisely let $f:(X,\Oo)\to (Y,\Oo')$ be a morphism of ringed sites. Then the flat model structure makes both adjunctions, $f^*\dashv f_*$ and $\otimes\dashv \Shom$, into Quillen adjunctions. For ringed spaces this is originally due to Gillespie.
	
	This gives a proof of the existence of $\otimes^L$ and $R\Shom$, which does not rely on triangulated categories. In particular we do not want to use $\Tor$-functors during the construction. We assume sites to be small.
	
	\begin{prop}
		Let $(X,\Oo)$ be a ringed site. Denote by $\Ff$ the class of flat $\Oo$-modules. Then $(\Ff,\Ff^{\perp})$ on $\Sh(X,\Oo)$ induces a model structure (see \cref{emp:induces-model-structure}).
	\end{prop}
		\begin{proof}
			See for the \cite{Gillespie06} for the case of ringed topological spaces and \cite[Section 2.5]{Recktenwald2019} for the generalization to (small) ringed sites (not necessarily with enough points).
		\end{proof}
	Elements in $\Cc:=\Ff^\perp$ are called \emph{cotorsion}. The induced model category structure is called the \emph{flat model structure}. 
	
	From now on categories of chain complexes of sheaves are assumed to be equipped with the flat model structure.
	\begin{lemma}\label{lem:flat-is-tensor-x-split}
		Let $(X,\Oo)$ be a ringed site, $F$ a flat $\Oo$-module and $X$ any $\Oo$-module. If
		\[
			0\to A\to B\to F\to 0
		\]
		is exact, then so is
		\[
			0\to X\otimes A\to X\otimes B\to X\otimes F\to 0.
		\]
	\end{lemma}
		\begin{proof}
			Let $0\to K\to G\to X\to 0$ be a short exact sequence where $G$ is flat. Then the result follows by applying the snake lemma to
			\[
				\begin{tikzcd}
					&&& 0\ar[d] \\
					& K\otimes A\ar[r]\ar[d] & K\otimes B \ar[r]\ar[d] & K\otimes F\ar[r]\ar[d] & 0 \\
					0 \ar[r] & G\otimes A\ar[r]\ar[d] & G\otimes B \ar[r]\ar[d] & G\otimes F\ar[r]\ar[d] & 0\\
					&X\otimes A\ar[r]\ar[d] & X\otimes B\ar[r]\ar[d] & X\otimes F\ar[d]\ar[r] & 0\\
					& 0 & 0 & 0
				\end{tikzcd}
			\]
		\end{proof}
	\begin{thm}\label{thm:pullback-pushforward-quillen}
		Let $f:(X,\Oo)\to (Y,\Oo')$ be a morphism of ringed sites. Then 
		\[
			\Ch(f^*):\Ch(\Sh(Y,\Oo'))\rightleftarrows \Ch(\Sh(X,\Oo)):\Ch(f_*)
		\]
		is a Quillen adjunction.
	\end{thm}
		\begin{proof}
			We check the conditions of \sref{Theorem}{thm:criterion-for-quillen-adjunction}. 
			\begin{enumerate}
				\item
					Let $F\in \Ff$ be a flat $\Oo'$-module and 
					\[
						0\to A\to B\to F\to 0
					\]
					be an exact sequence. 
					We have to show that
					\[
						0\to f^*A \to f^* B \to f^*F \to 0 
					\]
					is exact.
					Recall that $f^*X=\pb X\otimes_{\pb \Oo'} \Oo$. Because $\pb$ is an exact functor, the sequence
					\[
						0\to \pb A\to \pb B\to \pb F\to 0
					\]
					is exact.\\
					It remains to show that $\pb F$ is a flat $\pb \Oo'$-module on $X$, then the result follows from the previous lemma. In the case of ringed topological spaces, or sites with enough points, flatness can be checked on stalks and $(\pb F)_x=F_{f(x)}$ is a flat $(\pb \Oo')_x=\Oo'_{f(x)}$-module. For the general case the argument is more involved, see \cite[Tag 05VD]{stacks-project}.
				\item
					Let $F\in \Ff$ be a flat $\Oo'$ module and
					\[
						0\to A\to B\to C\to 0
					\]
					be a short exact sequence of $\Oo$-modules. It remains exact when we consider it as a sequence of $f^{-1}\Oo'$-modules. Note that 
					\[
						f^*F\otimes_\Oo A=f^{-1}F\otimes_{f^{-1}\Oo'}\Oo\otimes_\Oo A\cong f^{-1}F\otimes_{f^{-1}\Oo'}A
					\]
					and that
					\[
						0\to f^{-1}F\otimes_{f^{-1}\Oo'}A\to f^{-1}F\otimes_{f^{-1}\Oo'}B\to f^{-1}F\otimes_{f^{-1}\Oo'}C\to 0
					\]
					is exact because $f^{-1}F$ is a flat $f^{-1}\Oo'$ module. It follows that $f^*F$ is flat.
			\end{enumerate}

		\end{proof}
		
	\begin{coro}\label{cref:final-coro}
		Let $(X,\Oo)$, $(Y_i,\Oo'_i)$ (for $i=1,\ldots, n)$ be ringed site and let $f_i:(X,\Oo)\to (Y_i,\Oo'_i)$ be a collection of morphisms. Then
		\[
			f_1^*-\otimes\ldots\otimes f_n^*-: \Ch(\Sh(Y_1,\Oo_1'))\times \ldots\times \Ch(\Sh(Y_n,\Oo_n'))\to \Ch(\Sh(X,\Oo))
		\]
		is a Quillen adjunction in $n$ variables.
	\end{coro}
		\begin{proof}
			The construction of the right adjoints is omitted.
			We check the conditions in \sref{Theorem}{thm:cot-main}. 
			\begin{enumerate}
				\item 
					Fix some $j\in \set{1,\ldots, n}$ and let $F_i$ be flat $\Oo_i'$-modules for $i\neq j$. We have to show that when 
					\[
						0\to A_j\to B_j\to F_j\to 0
					\]
					is exact in $\Sh(Y_i,\Oo_i')$, then so is
					\[
						0\to f_1^*F_1\otimes\ldots f_j^*A_j\ldots\otimes f_n^*F_n\to f_1^*F_1\otimes\ldots f_j^*B_j\ldots\otimes f_n^*F_n\to f_1^*F_1\otimes\ldots f_j^*F_j\ldots\otimes f_n^*F_n\to 0
					\]
					In the proof of \cref{thm:pullback-pushforward-quillen} we have shown that
					\[
						0\to f_j^*A_j\to f_j^*B_j\to f_j^*F_j\to 0
					\]
					is exact and $f_j^*F_j$ is a flat $\Oo$-module. Hence the claim follows form \sref{Lemma}{lem:flat-is-tensor-x-split}.
				\item 
					If $F_i$ is flat $\Oo_i'$-module for all $i$, then $f_i^*F_i$ is a collection of flat $\Oo$-modules, therefore the tensor product $f_1^*F_1\otimes~\ldots f_j^*F_j\ldots\otimes~f_n^*F_n$ is also flat. 
			\end{enumerate}
		\end{proof}
		
	\begin{emp}
		\emph{Fibrations} and \emph{opfibrations} are a natural framework to study families of (model) categories. One can define a suitable \multi{category} of all complexes of modules on ringed sites, which is then fibred and opfibred over the category of ringed sites (more precisely over $\Rsite^{op}$). The previous Corollary then states that this lifts to a bifibration of \multimodelcategories{} in the sense of \cite[Definition 5.1.3]{Hoermann2017}. See \cite{Recktenwald2019} for details.
	
	\end{emp}

\newpage
\thispagestyle{empty}
\raggedright{
	\bibliographystyle{alpha}
	\bibliography{CotPairs.bib}
}


\end{document}

%% file: Introduction.tex
Let $\Aa$ be an abelian category and let $(\Dd,\Ee)$ be a \emph{cotorsion pair}, i.e. classes of objects such that 
\begin{itemize}
	\item $D\in \Dd$ if and only if $\Ext^1(D,E)=0$ for all $E\in \Ee$;
	\item $E\in \Ee$ if and only if $\Ext^1(D,E)=0$ for all $D\in \Dd$.
\end{itemize}
In \cite{Hovey02} Hovey proves a close relationship between model structures on $\Ch(\Aa)$ and cotorsion pairs on $\Aa$ - each abelian model structure gives rise to two \enquote{compatible} cotorsion pairs on $\Ch(\Aa)$ and (under mild assumptions) vice versa.
In \cite{Gillespie04} and \cite{Gillespie06} Gillespie showed that in many cases it is sufficient to consider one cotorsion pair on $\Aa$, instead of $\Ch(\Aa)$, to obtain an abelian model structure on chain complexes.
He then used this to construct the \emph{flat model structure} for sheaves on a ringed space $(X,\Oo)$ and show that it is monoidal, i.e. it makes the adjunction of two variables $\otimes\dashv \Shom$ into a Quillen adjunction.

It should be noted that the tensor product has two special properties, which cannot be expected solely from the fact that it is a left adjoint:
\begin{enumerate}
	\item When $F$ is flat then $F\otimes - $ is an exact functor.
	\item
		Let $F$ be flat, $M$ be any $\Oo$-module and assume the sequence 
		\[
			0\to A\to B\to F\to 0
		\]
		is exact, then
		\[
			0\to M\otimes A\to M\otimes B\to M\otimes F\to 0
		\]
		is also exact.
\end{enumerate}
Consider, however, the following natural situation. Let $f:(X,\Oo)\to (Y,\Oo')$ be a morphism of ringed spaces. Then there is a functor
\[
	-\otimes f^*-:\Sh(X,\Oo)\times \Sh(Y,\Oo')\to \Sh(X,\Oo)
\]
which fails both of the above properties.

The purpose of this paper is to generalize and simplify results for the $\otimes\dashv \Shom$ adjunction to general adjunctions of $n$-variables. We prove the following (see \cref{thm:cot-main} for the precise statement)
\begin{thm*}
	Let $\Aa_0,\ldots, \Aa_n$ be abelian categories with cotorsion pairs $(\Dd_i,\Ee_i)$ and let 
	\[
		F:\Aa_1\times\ldots\times \Aa_n\to \Aa_0
	\]
	be an $n$-variable left adjoint. Assume the $\Aa_i$ and $(\Dd_i,\Ee_i)$ satisfy some mild conditions and
	\begin{enumerate}[label=(\alph*)]
		\item Let $1\leq j\leq n$ be arbitrary. For each short exact sequence in $\Aa_j$ of the form
			\[
				0\to A_j \to B_j\to \underbrace{D_j}_{\in \Dd_j}\to 0
			\]
			the sequence
			\[
				0\to F(D_1,\ldots, A_j,\ldots, D_n) \to F(D_1,\ldots, B_j,\ldots, D_n)\to F(D_1,\ldots, D_j,\ldots, D_n)\to 0
			\]
			is exact for each collection $D_i\in \Dd_i$.
		\item
			For each collection of objects $D_i\in \Dd_i$ it is $F(D_1,\ldots,D_n)\in \Dd_0$.
	\end{enumerate}
	Then $F$ induces a Quillen adjunction of $n$-variables on chain complexes.
\end{thm*}
Note that both conditions can be checked on the underlying abelian category, instead of the category of chain complexes. 

In \cref{sec:applications} we reprove Gillespie's result that the flat model structure in monoidal. Also we show that $f^*\dashv f_*$ is a Quillen adjunction for the flat model structure, which to the best of our knowledge has not appeared in the literature yet - even though it has probably been known before.

In our thesis we \cite{Recktenwald2019} use this result to show that there is a general derivator four-functor formalism for ringed sites in the language of \cite{Hormann2018}. We will then use work by H\"{o}rmann (\cite{1701.02152}, \cite{1902.03625}) to extend this to concrete examples of derivator six functor formalisms.

We start by introducing the reader to the necessary notions in \sref{Section}{sec:setting-and-notation}. 

Then we consider the one variable case.

In \sref{Section}{sec:n-var-case} we lay down the basic definitions for adjunctions of $n$ variables, generalize Hovey's criterion \cite[Theorem 7.2]{Hovey02} from the case of the tensor product to general adjunctions, and prove our main result.

%% file: SettingNotation.tex
In this section we recall the basic notions and results on cotorsion pairs and abelian model categories. We assume that the reader is familiar with model categories and Quillen adjunctions (see \cite{Hovey99}). For a more thorough treatment of cotorsion pairs, see \cite{EnochsJenda00}.

For all abelian categories $\Aa$ we will always assume that $\Ext^n_\Aa(A,B)$ is a set for all $n\in \N$ and all objects $A$ and $B$ in $\Aa$. This is in particular the case if the category has enough projective or injective objects.

\subsection{Cotorsion Pairs}
	
	\begin{defi}\label{defi:cot-pair}
		Let $\Aa$ be an abelian category and $\Dd$, $\Ee$ be classes of objects. 
		We write
		\begin{align*}
		\Dd^\perp &= \set{X\in \Aa\where \Ext^1(D,X)=0 \text{ for all } D\in \Dd} \text{ and } \\
		{}^\perp\Ee &= \set{X\in \Aa\where \Ext^1(X,E)=0 \text{ for all } E\in \Ee}.
		\end{align*}
		We say $(\Dd,\Ee)$ is a \emph{cotorsion pair} if $\Dd^\perp=\Ee$ and ${}^\perp\Ee=\Dd$. We define the following properties of cotorsion pairs:
		\begin{enumerate}
			\item
			If there is a set $\Ss\subset \Dd$ such that $\Ss^\perp=\Ee$, we say that $(\Dd,\Ee)$ is \emph{cogenerated by a set}.
			\item 
			If for any object $X$ of $\Aa$ there is a short exact sequence
			\[
			0\to X\to E\to D\to 0
			\]
			with $D\in \Dd$ and $E\in \Ee$, then we say that $(\Dd,\Ee)$ has \emph{enough injectives}.
			\item
			If for any object $X$ of $\Aa$ there is a short exact sequence
			\[
			0\to E\to D\to X\to 0
			\]
			with $D\in \Dd$ and $E\in \Ee$, then we say that $(\Dd,\Ee)$ has \emph{enough projectives}.
			\item
			If $(\Dd,\Ee)$ has enough projectives and enough injectives, then we call the cotorsion pair \emph{complete}.
			\item 
			If for all $n\geq 1$ and all $D\in \Dd$ and $E\in \Ee$ it is
			\[
			\Ext^n(D,E)=0,
			\]
			then $(\Dd,\Ee)$ is called \emph{hereditary}.
		\end{enumerate}
	\end{defi}
	
	It is clear that if $(\Dd,\Ee)$ is a cotorsion pair every projective object of $\Aa$ is in $\Dd$ and dually every injective object is in $\Ee$. Also if $(\Dd,\Ee)$ is a cotorsion pair on $\Aa$, then $(\Ee,\Dd)$ is a cotorsion pair on $\Aa^{op}$.
	
	\begin{example}
		For every abelian category $\Aa$ the pair $(\Ob(\Aa),\text{injectives})$ is a cotorsion pair with enough projectives, as
		\[
		0\to 0\to X\to X\to 0
		\]
		shows. It has enough injectives if and only if the category $\Aa$ has enough injectives.
	\end{example}
	\begin{example}
		Let $R$ be a ring and denote by $\Ff$ the class of flat modules over $R$. Then $(\Ff,\Ff^\perp)$ is a complete and hereditary cotorsion pair, which is cogenerated by a set. None of these statements is trivial. See for example \cite[Theorem 1.1]{EnochsOyonarte01}. 
		\\
		The same is true for ringed topological spaces (see \cite[Section 4]{Gillespie06}) and even for small ringed sites (see \cite[Section 2.5]{Recktenwald2019})
	\end{example}
	
	\begin{emp}
		There are several different properties, which are usually referred to as having \enquote{enough objects of a certain type}:
		\begin{enumerate}
			\item $\Aa$ has enough projectives (resp.\@ injectives).
			\item $(\Dd,\Ee)$ has enough projectives (resp.\@ injectives).
			\item Every object of $\Aa$ is a quotient (resp.\@ subobject) of an object of $\Dd$ (resp.\@ $\Ee$).
		\end{enumerate}
		
		Note that clearly each of the first two properties implies the third. \cref{prop:enough-of-one-implies-complete} gives a partial reversal of this fact.
	\end{emp}
	
	\begin{lemma} 
		Let $(\Dd,\Ee)$ be a cotorsion pair. Then $\Dd$ and $\Ee$ are closed under extensions.
	\end{lemma}
	\begin{proof}
		Let 
		\[
		0\to D\to A\to D'\to 0
		\]
		be a short exact sequence with $D,D'\in \Dd$. For any $E\in \Ee$ we have an exact sequence
		\[
		\ldots\to \underbrace{\Ext^1(D',E)}_{=0}\to \Ext^1(A,E)\to \underbrace{\Ext^1(D,E)}_{=0}\to \ldots
		\]
		Hence $A\in \Dd$. The proof for $\Ee$ is similar.
	\end{proof}
	
	\begin{prop}\label{prop:enough-of-one-implies-complete} 
		Let $\Aa$ be an abelian category and $(\Dd,\Ee)$ a cotorsion pair with enough injectives (resp.\@ projectives).
		If every object in $\Aa$ is a quotient of an object in $\Dd$ (resp.\@ a subobject of an object in $\Ee$), then $(\Dd,\Ee)$ is complete.
	\end{prop}
	\begin{proof}
		Assume $(\Dd,\Ee)$ has enough injectives and that every object in $\Aa$ is a quotient of an object in $\Dd$. Let $A\in \Aa$ be an object. By assumption we find $D\to A\to 0$ with $D\in \Dd$. Let $K$ denote the kernel of $D\to A$, i.e.\@ 
		\[
		0\to K\to D\to A\to 0
		\]
		is a short exact sequence. By assumption we find $E\in \Ee$ and $D'\in \Dd$ to make this into an exact diagram
		\[
		\begin{tikzcd}
		& 0 \ar[d] & 0\ar[d]\\
		0\ar[r] & K\ar[r]\ar[d] & D\ar[r] \ar[d]& A\ar[r]\ar[d,equal] & 0\\
		0\ar[r]& E\ar[d]\ar[r] & \square\ar[d]\ar[r] & A\ar[r] & 0\\
		& D'\ar[d]\ar[r,equal] & D'\ar[d]\\
		& 0 & 0
		\end{tikzcd}
		\]
		where $\square$ denotes the pushout. Since $\Dd$ is closed under extension $0\to E\to \square\to A\to 0$ shows that $(\Dd,\Ee)$ has enough projectives.
	\end{proof}
	
	We will now study hereditary cotorsion pairs in more detail.
	
	\begin{defi}
		Let $(\Dd,\Ee)$ be a cotorsion pair.
		\begin{enumerate}
			\item If $\Dd$ is closed under kernels of epimorphisms, then $\Dd$ is called \emph{resolving}.
			\item If $\Ee$ is closed under cokernels of monomorphisms, then $\Ee$ is called \emph{coresolving}.
		\end{enumerate}
	\end{defi}
	
	\begin{lemma}\label{lem:ext2=0-implies-co-resolving} 
		Let $(\Dd,\Ee)$ be a cotorsion pair and assume $\Ext^2(D,E)=0$ for all $D\in \Dd$ and $E\in \Ee$. Then $\Dd$ is resolving and $\Ee$ is coresolving.
	\end{lemma}
	\begin{proof}
		Let $0\to A\to D\to D'\to 0$ be a short exact sequence with $D,D'\in \Dd$. For $E\in \Ee$ we find an exact sequence
		\[
		\ldots\to \underbrace{\Ext^1(D,E)}_{=0}\to \Ext^1(A,E)\to \underbrace{\Ext^2(D',E)}_{=0}\to \ldots
		\]
		Hence $\Ext^1(A,E)=0$ for all $E\in \Ee$ and we conclude $A\in \Dd$. The dual argument shows that $\Ee$ is coresolving.
	\end{proof}
	In particular if $(\Dd,\Ee)$ is hereditary, then $\Dd$ is resolving and $\Ee$ is coresolving.
	\begin{lemma}\label{lem:hereditary-criterion-II}
		Let $\Aa$ be an abelian category and $(\Dd,\Ee)$ a cotorsion pair.
		\begin{enumerate}
			\item Assume $\Aa$ has enough injectives and $\Ee$ is coresolving. Then $(\Dd,\Ee)$ is hereditary.
			\item Assume $\Aa$ has enough projectives and $\Dd$ is resolving. Then $(\Dd,\Ee)$ is hereditary.
		\end{enumerate}
	\end{lemma}
	\begin{proof}
		We show $\Ext^n(D,E)=0$ by induction on $n$. The case $n=1$ follows by assumption. Let $E\in \Ee$ be arbitrary and let 
		\[
		0\to E\to I\to E'\to 0
		\]
		be a short exact sequence with $I$ injective. Since $\Ee$ is coresolving and $I\in \Ee$ it is $E'\in \Ee$. For any $D\in \Dd$ we find an exact sequence
		\[
		\ldots\to \Ext^n(D,E')\to \Ext^{n+1}(D,E)\to \Ext^{n+1}(D,I)\to \ldots
		\]
		By induction $\Ext^n(D,E')=0$ and because $I$ is injective $\Ext^{n+1}(D,I)=0$. Hence $\Ext^{n+1}(D,E)=0$ as desired. 
		
		The second statement is dual.
	\end{proof}
	
	\begin{lemma}[{\cite[Lemma 2.3]{Gillespie2016}}]\label{lem:hereditary-criterion}
		Let $(\Dd,\Ee)$ be a cotorsion pair. 
		\begin{enumerate}
			\item If $\Dd$ is resolving and each object of $\Aa$ is a quotient of an object in $\Dd$, then $\Ext^2(D,E)=0$ for all $D\in \Dd$ and $E\in \Ee$.
			\item If $\Ee$ is coresolving and each object of $\Aa$ is a subobject of an object in $\Ee$, then $\Ext^2(D,E)=0$ for all $D\in \Dd$ and $E\in \Ee$.
		\end{enumerate}
	\end{lemma}
	\begin{proof}
		We only show the first statement as the other one is dual.\\
		Let $D\in \Dd$ and $E\in \Ee$ and fix an exact sequence
		\[
		0\to E\to A\to B\to D\to 0.
		\]
		We write $X=\im (A\to B)=\ker(B\to D)$. By assumption we find an object $D'\in \Dd$ with an epimorphism $D'\to B$. Consider
		\[
		\begin{tikzcd}
		0\ar[r] & \tilde{D}\ar[r]\ar[d] & D'\ar[r]\ar[d] & D\ar[d,equal]\ar[r] & 0\\
		0\ar[r] & X \ar[r] & B\ar[r] & D\ar[r] & 0.
		\end{tikzcd}
		\]
		where $\tilde{D}$ denotes the pullback. By assumption $\tilde{D}\in \Dd$. Since $\tilde{D}\to X$ is surjective we obtain the following exact diagram
		\[
		\begin{tikzcd}
		0\ar[r]& E\ar[r]\ar[d,equal]& \tilde{A} \ar[r]\ar[d] & \tilde{D}\ar[r]\ar[d] & 0\\
		0\ar[r] & E\ar[r] & A\ar[r] & X\ar[r] & 0
		\end{tikzcd}
		\]
		Note that since $\tilde{D}\in \Dd$ and $E\in \Ee$ the top row splits. Via concatenation we see that the following extensions are equivalent in $\Ext^2(D,E)$.
		\[
		\begin{tikzcd}
		0\ar[r] & E\ar[r]\ar[d,equal] & \tilde{A}\ar[r]\ar[d] & D'\ar[r]\ar[d] & D\ar[r]\ar[d,equal] & 0\\
		0\ar[r] & E\ar[r]& A\ar[r] & B\ar[r] & D\ar[r]& 0
		\end{tikzcd}
		\] 
		However since the top row is built from a split short exact sequence it is trivial. To see this recall that concatenation distributes over sums, i.e.\@ $\Ext^1(\tilde{D},E)\times \Ext^1(D,\tilde{D})\to \Ext^2(D,E)$ is bilinear.
	\end{proof}
	
	\begin{coro}\label{coro:true-hereditary-criterion}
		Let $\Aa$ be a Grothendieck abelian category and $(\Dd,\Ee)$ a cotorsion pair. If $\Dd$ is resolving and every object of $\Aa$ is a quotient of an object of $\Dd$, then $(\Dd,\Ee)$ is hereditary.
	\end{coro}
	\begin{proof}
		By \cref{lem:hereditary-criterion} we see that $\Ext^2(D,E)=0$ for any $D\in \Dd$ and $E\in \Ee$. Therefore, by \cref{lem:ext2=0-implies-co-resolving} $\Ee$ is coresolving. Then \cref{lem:hereditary-criterion-II} implies the desired result.
	\end{proof}

\subsection{Abelian model categories}

	In this section we will explain the relationship between cotorsion pairs and model structures on chain complexes. This connection has been established in \cite{Hovey02} and was then refined in \cite{Gillespie04} and \cite{Gillespie06}. We refer interested readers, who want to learn more about this specific topic, to the nice introduction in \cite{Hovey07}.
	
	\begin{defi}
		Let $\Aa$ be an abelian category. A model structure on $\Ch(\Aa)$ is called abelian if the weak equivalences $\Ww$ are the quasi-isomorphisms and
		\begin{itemize}
			\item cofibrations are precisely monomorphisms with cofibrant cokernel;
			\item fibrations are precisely epimorphisms with fibrant kernel.
		\end{itemize}
	\end{defi}
	In an abelian model category $\Ch(\Aa)$ the acyclic complexes are precisely the objects, that are weakly equivalent to the zero object. We will denote them by $\Ac$.
	
	\begin{lemma}\label{lem:trivial-cofib}
		Let $\Ch(\Aa)$ be an abelian model category. Then
		\begin{enumerate}
			\item trivial cofibrations are precisely monomorphisms with acyclic cofibrant cokernel;
			\item trivial fibrations are precisely epimorphisms with acyclic fibrant kernel.
		\end{enumerate}
	\end{lemma}
	\begin{proof}
		We only show the first statement. Let $f:A^\bu\to B^\bu$ be a trivial cofibration. By definition this means that there is a short exact sequence
		\[
		0\to A^\bu\to B^\bu\to C^\bu\to 0
		\]
		where $C^\bu$ is cofibrant. We apply the long exact sequence of homology to see that $C^\bu$ is exact. \\
		On the other hand, given a monomorphism $f:A^\bu\to B^\bu$ with exact cokernel the long exact sequence shows that $f$ is a quasi-isomorphism.
	\end{proof}

	\begin{thm}[{Hovey \cite[Theorem 2.2]{Hovey02}}]\label{thm:cot-pairs-and-model-cats}
		Let $ \Ch(\Aa)$ be an abelian model category. Denote by $\Ac$ the class of acyclic complexes, by $\Dd$ the class of cofibrant objects and by $\Ee$ the class of fibrant objects. Then $(\Dd\cap \Ac,\Ee)$ and $(\Dd,\Ee\cap \Ac)$ are complete cotorsion pairs on $\Ch(\Aa)$.
		
		Conversely given classes $\Dd,\Ee$ such that $(\Dd\cap \Ac,\Ee)$ and $(\Dd,\Ee\cap\Ac)$ are complete cotorsion pairs on $\Ch(\Aa)$, then there exists an abelian model category structure on $\Ch(\Aa)$ where $\Dd$ are the cofibrant, and $\Ee$ are the fibrant objects.
	\end{thm}
	
	In light of the above theorem we want to start with a cotorsion pair $(\Dd,\Ee)$ on an abelian category $\Aa$ and produce two complete cotorsion pairs on $\Ch(\Aa)$ in order to induce a model structure. We begin with some definitions.
	
	\begin{defi}[{\cite[Definition 2.2]{Gillespie06}}]\label{defi:types-of-complexes}
		Let $(\Dd,\Ee)$ be a cotorsion pair on an abelian category $\Aa$. Let $A^\bu$ be a complex. We write $Z^nA^\bu=\ker(d^n:A^n\to A^{n+1})$ as usual.
		\begin{enumerate}
			\item 
			$A^\bu$ is called a \emph{$\Dd$-complex} if it is exact and $Z^nA^\bu\in \Dd$ for all $n$. We write $\tilde{\Dd}$ for the class of $\Dd$-complexes.
			\item 
			$A^\bu$ is called an \emph{$\Ee$-complex} if it is exact and $Z^nA^\bu\in \Ee$ for all $n$. We write $\tilde{\Ee}$ for the class of $\Ee$-complexes.
			\item 
			$A^\bu$ is called a \emph{$\dg\Dd$-complex} if $A^n\in \Dd$ for all $n$, and every map $f^\bu:A^\bu\to E^\bu$ to a $\Ee$-complex $E^\bu$ is homotopic to zero. We write $\dg\tilde{\Dd}$ for the class of $\dg\Dd$-complexes.
			\item 
			$A^\bu$ is called a \emph{$\dg\Ee$-complex} if $A^n\in \Ee$ for all $n$, and every map $f^\bu:D^\bu\to A^\bu$ from a $\Dd$-complex $D^\bu$ is homotopic to zero. We write $\dg\tilde{\Ee}$ for the class of $\dg\Ee$-complexes.
		\end{enumerate}
	\end{defi}
	
	\begin{emp}
		It is well-known that an exact complex of the form
		\[
			P^\bu=\ldots\to P^{-1}\to P^0\to 0\to 0\to \ldots
		\]
		where each $P^i$ is projective, is a projective object in the category of complexes.
		
		In particular resolutions of this type can be used to compute derived functors. However, an unbounded exact complex of projective objects does not need to be a projective object in general as the following example shows:
		
		Let $R=\Z/4\Z$ and 
		\[
			P^\bu=\begin{tikzcd}
			\ldots\ar[r] & R\ar[r,"\cdot 2"]& R\ar[r,"\cdot 2"]& R\ar[r] & \ldots
			\end{tikzcd}
		\]
		It is easy to check that the identity is not homotopic to zero, i.e.\@ $P^\bu$ is not contractible. However, it can be shown that projective objects in $\Ch(\Aa)$ are precisely the contractible complexes of projectives. 
		
		The above definition remedies this. Assume $\Aa$ has enough injectives and let $\Pp$ denote the class of projective objects in $\Aa$. Then $(\Pp,\Ob(\Aa))$ is a hereditary cotorsion pair and, as we will see, by \cref{lem:hereditary-for-compatible} a $\Pp$-complex is precisely an exact $\dg\Pp$-complex. Therefore $\Pp$-complexes are precisely the projective objects in $\Ch(\Aa)$. To see this, note that $\Ob(\Aa)$-complexes are just the exact complexes, hence for a $\Pp$-complex $P^\bu$, the identity is null homotopic.
	\end{emp}
	
	\begin{emp}\label{emp:note-on-quillen-adjunctions}
		In abelian model categories we can give the following simple reformulation of the definition of Quillen adjunctions.
		Let $F:\Ch(\Aa)\to \Ch(\Bb)$ and $G:\Ch(\Bb)\to \Ch(\Aa)$ be additive functors between abelian model categories and let $F$ be left adjoint to $G$. We use notation as in \cref{thm:cot-pairs-and-model-cats}.
		Then $F$ is left Quillen if and only if the following three conditions hold:
		\begin{enumerate}
			\item If $0\to A^\bu\to B^\bu\to D^\bu\to 0$ is a short exact sequence and $D^\bu \in \Ac\cap \Dd$, then the sequence $0\to F(A^\bu)\to F(B^\bu) \to F(D^\bu)\to 0$ is exact.
			\item If $D^\bu\in \Ac\cap \Dd$, then $F(D^\bu)\in \Ac\cap \Dd$.
			\item If $D^\bu\in \Dd$ then $F(D^\bu)\in \Dd$.
		\end{enumerate}
	\end{emp}

	\begin{lemma}[{\cite[Lemma 3.4]{Gillespie04}}]\label{lem:bounded-are-dg}
		Let $(\Dd,\Ee)$ be a cotorsion pair. 
		\begin{itemize}
			\item Complexes bounded in the direction of the arrows with entries in $\Dd$ are $\dg\Dd$-complexes.
			\item Complexes bounded against the direction of the arrows with entries in $\Ee$ are $\dg\Ee$-complexes.
		\end{itemize}
	\end{lemma}

	\noindent The following part explains why different properties of a cotorsion pair on $\Aa$ are needed in order to induce compatible cotorsion pairs on $\Ch(\Aa)$. 
	
	\begin{prop}[{\cite[Lemma 3.5 and Proposition 3.6]{Gillespie04}}]\label{prop:enough-tilde-objects}
		Let $(\Dd,\Ee)$ be a cotorsion pair on $\Aa$. If every object of $\Aa$ is a quotient of an object of $\Dd$ and a subobject of an object of $\Ee$ then 
		\[
		(\dg \tilde{\Dd},\tilde{\Ee}) \text{ and } (\tilde{\Dd},\dg\tilde{\Ee})
		\]
		are cotorsion pairs on $\Ch(\Aa)$ and every object of $\Ch(\Aa)$ is a quotient of an object of $\tilde{\Dd}$ and a subobject of an object of $\tilde{\Ee}$.
	\end{prop}

	\begin{lemma}[{\cite[Lemma 3.10]{Gillespie04}}] 
		For every cotorsion pair $(\Dd,\Ee)$ it is $\tilde{\Dd}\subset\dg\tilde{\Dd}\cap \Ac$ and $\tilde{\Ee}\subset\dg\tilde{\Ee}\cap \Ac$.
	\end{lemma}
	In order to induce a model structure we need $\tilde{\Dd}=\dg\tilde{\Dd}\cap \Ac$ and $\tilde{\Ee}=\dg\tilde{\Ee}\cap \Ac$, in which case we will call the pairs \emph{compatible}. This is where we will need $(\Dd,\Ee)$ to be hereditary.

	\begin{lemma}[{\cite[Theorem 3.12]{Gillespie04}}]\label{lem:hereditary-for-compatible} 
		Let $(\Dd,\Ee)$ be a hereditary cotorsion pair in $\Aa$. 
		\begin{enumerate}
			\item
			If $\Aa$ has enough projectives, then $\tilde{\Ee}=\dg\tilde{\Ee}\cap \Ac$. 
			\item
			If $\Aa$ has enough injectives, then $\tilde{\Dd}=\dg\tilde{\Dd}\cap \Ac$.
		\end{enumerate}
	\end{lemma}
	
	\begin{lemma}[{\cite[Lemma 3.14]{Gillespie04}}]\label{lem:compatible-criterion} 
		Let $(\Dd,\Ee)$ be a cotorsion pair in $\Aa$ (not necessarily hereditary). 
		\begin{enumerate}
			\item
			If $(\dg\tilde{\Dd},\tilde{\Ee})$ is a cotorsion pair with enough injectives and $\dg\tilde{\Dd}\cap \Ac=\tilde{\Dd}$, then $\dg\tilde{\Ee}\cap \Ac= \tilde{\Ee}$.
			\item
			If $(\tilde{\Dd},\dg\tilde{\Ee})$ is a cotorsion pair with enough projectives and $\dg\tilde{\Ee}\cap \Ac=\tilde{\Ee}$, then $\dg\tilde{\Dd}\cap \Ac=\tilde{\Dd}$.
		\end{enumerate}
	\end{lemma}
	\begin{prop}[{\cite[Lemma 3.6]{Gillespie06}}]\label{prop:enough-inj-crit}
		Let $(\Dd,\Ee)$ be a cotorsion pair in a Grothendieck abelian category $\Aa$ where each object is a quotient of an object of $\Dd$. If $(\Dd,\Ee)$ is cogenerated by a set, then the cotorsion pair of complexes $(\dg\tilde{\Dd},\tilde{\Ee})$ has enough injectives.
	\end{prop}
	\begin{rem}
		The previous proposition only uses the assumption that every object of $\Aa$ is a quotient of an object in $\Dd$ to ensure that $(\dg\tilde{\Dd},\tilde{\Ee})$ is a cotorsion pair.
	\end{rem}

	\begin{thm}[{Gillespie \cite[Corollaries 3.7 and 3.8]{Gillespie06}}]\label{thm:assumptions-give-model-cat}
		Let $\Aa$ be a Grothendieck abelian category with a hereditary cotorsion pair $(\Dd,\Ee)$, which is cogenerated by a set. Assume that every object of $\Aa$ is a quotient of an object of $\Dd$.
		Then
		\begin{enumerate}
			\item $\dg\tilde{\Dd}\cap \Ac=\tilde{\Dd}$ and $\dg\tilde{\Ee}\cap \Ac=\tilde{\Ee}$, and
			\item $(\dg\tilde{\Dd},\tilde{\Ee})$ is complete.
		\end{enumerate}
		If additionally $(\tilde{\Dd},\dg\tilde{\Ee})$ has enough injectives, then $(\tilde{\Dd},\dg\tilde{\Ee})$ also has enough projectives.
		Hence in that case we have an induced abelian model category structure on $\Ch(\Aa)$ where the weak equivalences are quasi-isomorphisms, the cofibrant objects are $\dg\tilde{\Dd}$, and the fibrant objects are $\dg\tilde{\Ee}$.
	\end{thm}

	\begin{rem}
		Note that in the previous theorem only the condition that $(\tilde{\Dd},\dg\tilde{\Ee})$ has enough injectives is a condition of chain complexes. All the other conditions are checked on $\Aa$.
	\end{rem}

	\begin{emp}
		\cref{thm:assumptions-give-model-cat} gives conditions on a cotorsion pair on $\Aa$ in order to induce a model structure on $\Ch(\Aa)$. It is natural to ask which of these conditions are implied by the result, i.e.\@ when we assume that $(\Dd,\Ee)$ induces a model structure.
	\end{emp}
	\begin{thm}\label{thm:assumptions}
		Let $\Aa$ be a Grothendieck abelian category with a cotorsion pair $(\Dd,\Ee)$. Assume that $(\dg\tilde{\Dd},\tilde{\Ee})$ and $(\tilde{\Dd},\dg\tilde{\Ee})$ are compatible complete cotorsion pairs. Then
		\begin{enumerate}
			\item $(\Dd,\Ee)$ is complete;
			\item every object of $\Aa$ is a quotient of an object in $\Dd$;
			\item $(\Dd,\Ee)$ is hereditary
		\end{enumerate}
	\end{thm}

	\begin{proof}
		Let $A$ be any object of $\Aa$. Denote by $S^0A$ the complex with $A$ in degree $0$ and zero everywhere else. There is a short exact sequence
		\[
		0\to S^0A\to \underbrace{E^\bu}_{\in \tilde{\Ee}}\to \underbrace{D^\bu}_{\in \dg\tilde{\Dd}} \to 0.
		\]
		Since $\tilde{\Ee}\subset \dg\tilde{\Ee}$ it is $E^0\in \Ee$ and we find
		\[
		0\to A\to \underbrace{E^0}_{\in {\Ee}}\to \underbrace{D^0}_{\in {\Dd}} \to 0
		\]
		which shows that $(\Dd,\Ee)$ has enough injectives. A similar argument shows that it also has enough projectives.
		The second statement is implied by the first \textit{a fortiori}.
		For the third statement it is enough to show that $\Dd$ is resolving by
		\cref{coro:true-hereditary-criterion}.

		Let $d^0:D^0\to D^1$ be an epimorphism with kernel $K$. We then find an epimorphism $D^{-1}\to K$ and we denote the composition $D^{-1}\to K\to D^0$ by $d^{-1}$. By continuing in this way we can construct an exact complex
		\[
		D^\bu=\ldots \to D^{-1}\to D^0\to D^1\to 0\to \ldots
		\]
		such that $D^i\in \Dd$ for all $i$ and $D^i=0$ for $i>1$. By \cref{lem:bounded-are-dg} $D^\bu$ is a $\dg{\Dd}$-complex and therefore by compatibility $D^\bu\in \tilde{\Dd}$. In particular $Z^0D^\bu=\ker d^0\in \Dd$.
	\end{proof}

%% file: 1VariableCase.tex
	This section gives a criterion for an adjunction between abelian categories to induce a Quillen adjunction on their chain complexes. It will then serve as the base case for a prove by induction of the $n$-variable case.
	
	\begin{emp}\label{emp:induces-model-structure}
		We have seen that, under certain conditions, a cotorsion pair $(\Dd,\Ee)$ on an abelian category will induce a model structure on $\Ch(\Aa)$. From now on we will simply say that \emph{$(\Dd,\Ee)$ induces a model structure} if $(\dg\tilde{\Dd},\tilde{\Ee})$ and $(\tilde{\Dd},\dg\tilde{\Ee})$ are compatible complete cotorsion pairs. 
		In particular this will allow us to use the properties implied by \cref{thm:assumptions} for $(\Dd,\Ee)$ itself.
	\end{emp}

	The following definition is motivated by the observation \cref{emp:note-on-quillen-adjunctions}.

	\begin{defi}\label{defi:right-left-split}
		Let $F:\Aa\to\Bb$ be a right exact functor and $\Ss\subset\Ob(\Aa)$ a class of objects. We call it \emph{$F$-right split} if for all short exact sequences
		\[
			0\to X\to Y\to \underbrace{S}_{\in\Ss}\to 0
		\]
		the sequence
		\[
			0\to FX\to FY\to FS\to 0
		\]
		is also exact. \\
		Similarly if $G:\Bb\to \Aa$ is left exact we call $\Ss\subset\Ob(\Bb)$ a \emph{$G$-left split} class of objects if 
		\[
			0\to S\to X\to Y\to 0
		\]
		with $S$ in $\Ss$ implies the exactness of
		\[
			0\to FS\to FX\to FY\to 0.
		\]
		This is equivalent to saying that a monomorphism with cokernel in $\Ss$ gets mapped to a monomorphism by $F$ (and epimorphisms with kernel in $\Ss$ getting mapped to epimorphisms by $G$).
	\end{defi}
	\begin{example}
		Let $(X,\Oo)$ be a ringed space and $\Aa=\lmod{\Oo}$ the category of $\Oo$-modules. 
		\begin{enumerate}
			\item The class of flat modules is $(A\otimes -)$-right split for any $\Oo$-module $A$.
			\item The class of flabby modules is $\Gamma(X,-)$-left split.
		\end{enumerate}
	\end{example}
	Let now $F:\Aa\rightleftarrows \Bb:G$ be an adjunction $F\dashv G$ of Grothendieck abelian categories with cotorsion pairs $(\Dd,\Ee)$ on $\Aa$ and $(\Dd',\Ee')$ on $\Bb$.

	We will often use the following easy fact, and restate it here for the readers convenience.
	\begin{lemma}\label{lem:mono-epi-crit}
		Let $\Cc$ be a category and $f:X\to Y$ be a morphism in $\Cc$. Then
		\begin{enumerate}
			\item if there exists an epimorphism $p:Y'\to Y$ and a morphism $g:Y'\to X$ with $fg=p$, then $f$ is an epimorphism.
			\item if there exists a monomorphism $m:X\to X'$ and a morphism $g:Y\to X'$ with $gf=m$ then $f$ is a monomorphism.
		\end{enumerate}
	\end{lemma}
		\begin{proof}
			We only show the first statement, as the other one is dual.
			Let $a,b:A\rightrightarrows X$ be two morphisms with $af=bf$. Then of course it also is $afg=bfg$. By assumption $fg=p$ is an epimorphism, hence from $ap=bp$ we can conclude $a=b$, as desired.
		\end{proof}
	\begin{prop}\label{prop:equivalent-conditions-1-variable}
		Let $\Aa$ and $\Bb$ be abelian categories and $F:\Aa\to \Bb$ be a left adjoint with right adjoint $G$. Also let $(\Dd,\Ee)$ and $(\Dd',\Ee')$ be a cotorsion pair on $\Aa$ and $\Bb$, respectively. Assume that every object of $\Aa$ is a subobject of an object in $\Ee$, and that every object of $\Bb$ is a quotient of an object in $\Dd'$.\\
		The following pairs of conditions are equivalent
		
		\vspace{ 1 em}
		\begin{minipage}[l]{0.5\textwidth}
			\begin{enumerate}
				\item[$(1a)$] $\Dd$ is $F$-right split
				\item[$(1b)$] $F(\Dd)\subset \Dd'$
			\end{enumerate}
		\end{minipage}
		\begin{minipage}[r]{0.45\textwidth} 
			\begin{enumerate}
				\item[$(2a)$] $\Ee'$ is $G$-left split
				\item[$(2b)$] $G(\Ee')\subset \Ee$
			\end{enumerate}
		\end{minipage}
		\vspace{ 0.15 em }
	\end{prop}
	\begin{rem}
		This proposition can be viewed as a generalization of the well-known fact that the right adjoint of an exact functor preserves injectives, by applying it to the $(\Aa,\text{Injectives})$ cotorsion pair.
	\end{rem}
		\begin{proof}
			We first claim that $(1b)$ implies $(2a)$.\\
			Let $0\to E\to Y\to Z\to 0$ with $E\in \Ee'$. We have to show that applying $G$ preserves exactness. Because $\Aa$ has enough $\Dd$-objects, we find an object $D\in \Dd$ and an epimorphism $D\to GZ$, as indicated in the following diagram
			\[
				\begin{tikzcd}
					&&&D\ar[d,twoheadrightarrow]\ar[ld, dashed, "{\exists?}",swap]\\
					0\ar[r] & GE\ar[r] & GY \ar[r] & GZ.
				\end{tikzcd}
			\]
			When there is a lift $D\to GY$ of $D\to GZ$ along $GY\to~ GZ$, then it follows that the morphism $GY\to GZ$ is an epimorphism as well. The existence of a lift is itself equivalent to the surjectivity of $\Hom_\Aa(D,GY)\to~ \Hom_\Aa(D,GZ)$. We use the adjointness of $F$ and $G$ and get
			\[
				\begin{tikzcd}
					\Hom_\Aa(D,GY)\ar[r]\ar[d,"\cong"] & \Hom_\Aa(D,GZ)\ar[d,"\cong"]\\
					\Hom_\Bb(FD,Y)\ar[r] & \Hom_\Bb(FD,Z)\ar[r] & \Ext^1_\Bb(FD,E).
				\end{tikzcd}
			\]
			By assumption $FD\in\Dd'$ and $E\in \Ee'$, hence $\Ext^1(FD,E)=0$. This shows exactness of 
			\[
				0\to GE\to GY\to GZ\to 0.
			\]
			We now claim that $(1a)$ and $(2a)$ together imply
			\begin{equation}\label{eq:ext-adjunction}
				\Ext^1_\Bb(FD,E)\cong \Ext^1_\Aa(D,GE) \quad \forall D\in \Dd,\text{ }E\in \Ee'.
			\end{equation}
			To see this we fix an element of $\Ext^1_\Bb(FD,E)$
			\[
				0\to E\to Y\to FD\to 0.
			\]
			We apply the functor $G$ to this short exact sequence, and use the unit morphism $D\to GFD$ to construct
			\[
				\begin{tikzcd}
					0\ar[r] & GE\ar[d,equal]\ar[r] & GY\times_{GFD} D\ar[r]\ar[d] & D\ar[r]\ar[d] & 0\\
					0 \ar[r]& GE \ar[r] & GY\ar[r] & GFD\ar[r] & 0
				\end{tikzcd}
			\]
			which is an element in $\Ext^1_\Aa(D,GE)$. By naturality this induces a morphism 
			\[
				\Ext^1_\Bb(FD,E)\to \Ext^1_\Aa(D,GE).
			\]

			Similarly we find a morphism in the other direction. It is easy to see that these two maps are mutually inverse.
			
			Finally, we prove that $(1a)$, $(1b)$ and $(2a)$ imply $(2b)$. Let $E\in \Ee'$, then we have to show that $\Ext^1_\Aa(D,GE)=0$ for all $D\in \Dd$. By \cref{eq:ext-adjunction} this is the same as showing $\Ext^1_\Bb(FD,E)=0$ for all $D\in \Dd'$ which is true by $(1b)$.
			
			The other direction, i.e.\@ showing that $(2a)$ and $(2b)$ imply $(1a)$ and $(1b)$ is the dual argument.
		\end{proof}

	\begin{coro}\label{coro:exactness-of-products}
		Let $\Aa$ be an abelian category with a cotorsion pair $(\Dd,\Ee)$ such that every object is a quotient of an object in $\Dd$ and a subobject of an object in $\Ee$. Let $I$ be a set.
		\begin{enumerate}
			\item 
				For an $I$-indexed family of short exact sequences
				\[
					0\to A_i\to B_i\to \underbrace{D_i}_{\in \Dd}\to 0
				\]
				the direct sum
				\[
					0\to \bigoplus_I A_i\to \bigoplus_I B_i\to \bigoplus_I D_i\to 0
				\]
				is exact, and $\bigoplus_I D_i\in \Dd$.
			\item 
				For an $I$-indexed family of short exact sequences
				\[
					0\to \underbrace{E_i}_{\in \Ee}\to A_i\to B_i\to  0
				\]
				the product
				\[
					0\to \prod_I E_i\to \prod_I A_i\to \prod_I B_i\to 0
				\]
				is exact and $\prod_I E_i\in \Ee$.
		\end{enumerate}
	\end{coro}
		\begin{proof}
			Denote by $\Aa^I$ the product category. It is clear that it has a cotorsion pair $(\Dd^I,\Ee^I)$ where the classes consist of tuples point-wise in $\Dd$ or $\Ee$ respectively. There is the constant diagram functor $c:\Aa\to \Aa^I$ and clearly $c(\Dd)\in \Dd^I$ and $c(\Ee)\in \Ee^I$. Since $c$ is exact splitness conditions with respect to $c$ are empty. By definition we have adjunctions
			\[
				\bigoplus_I \dashv c\dashv \prod_I
			\]
			and the result now follows from \cref{prop:equivalent-conditions-1-variable}.
		\end{proof}
	\begin{rem}
		It is well-known that a category with enough projective objects has exact products. The above Corollary is a generalization of that fact. Indeed for any abelian category $\Aa$ there is the cotorsion pair $(\text{projectives},\Ob(\Aa))$, and if $\Aa$ has enough projectives the assumption of the Corollary are met.
	\end{rem}

	\begin{thm}\label{thm:criterion-for-quillen-adjunction}
		Let $F:\Aa\rightleftarrows \Bb:G$ be an adjunction $F\dashv G$. Let $(\Dd,\Ee)$ and $(\Dd',\Ee')$ be cotorsion pairs on $\Aa$ and $\Bb$ respectively. Assume that the cotorion pairs induce model structures on $\Ch(\Aa)$ and $\Ch(\Bb)$, respectively. Additionally assume the equivalent conditions
		
		\vspace{ 1 em}
		\begin{minipage}[l]{0.46\textwidth}
			\begin{enumerate}
				\item[$(1a)$] $\Dd$ is $F$ right split
				\item[$(1b)$] $F(\Dd)\subset \Dd'$
			\end{enumerate}
		\end{minipage}
		\begin{minipage}[r]{0.46\textwidth}
			\begin{enumerate}
				\item[$(2a)$] $\Ee'$ is $G$ left split
				\item[$(2b)$] $G(\Ee')\subset \Ee$
			\end{enumerate}
		\end{minipage}
		\vspace{ 0.6 em }
		
		\noindent Then $F$ and $G$ induce a Quillen adjunction $\Ch(F):\Ch(\Aa)\rightleftarrows\Ch(\Bb):\Ch(G)$.
	\end{thm}
		We will use the notions introduced in \sref{Definition}{defi:types-of-complexes}.
		\begin{proof}
			We show that $\Ch(F)$ preserves trivial cofibrations. Then the dual argument will show the that $\Ch(G)$ preserves trivial fibrations, which will complete the proof. For ease of notation, and since $\Ch(F)$ simply applies $F$ index-wise, we will only write $F$.
			
			Let $0\to A^\bu\to B^\bu\to D^\bu\to0$ be a trivial cofibration, i.e.\@ $D^\bu\in \tilde{\Dd}$.
			Exactness of this short exact sequence means that for all $i$
			\[
				0\to A^i\to B^i\to D^i\to 0
			\]
			is exact. Since $\tilde{\Dd}\subset\dg\tilde{\Dd}$ it is $D^i\in \Dd$. Hence by $(1a)$ $F$ preserves exactness for each $i$, so
			\[
				0\to FA^\bu\to FB^\bu\to FD^\bu\to 0
			\]
			is exact in $\Ch(\Bb)$.\\
			It remains to show that $FD\in \tilde{\Dd'}$. For $i\in \Z$ fixed we can write
			\[
				0\to Z^iD^\bu\to D^i\to B^iD^\bu\to 0
			\]
			and by exactness $B^iD^\bu=Z^{i+1}D^\bu\in \Dd$. We apply $F$ to this short exact sequence, as well as the analogue short exact sequence for $i+1$ to find the diagram
			\[
				\begin{tikzcd}
					&&&0\ar[d] \\
					0\ar[r] & F(Z^iD^\bu) \ar[d,"f"]\ar[r] & FD^i \ar[d,equal] \ar[r] & F(Z^{i+1}D^\bu)\ar[r]\ar[d] & 0\\
					0 \ar[r] & Z^iFD^\bu\ar[r] & FD^i \ar[r] & FD^{i+1}\ar[d]\\
					&&& F(Z^{i+2}D^\bu)\ar[d]\\
					&&& 0
				\end{tikzcd}
			\]
			By applying the five lemma to
			\[
				\begin{tikzcd}
					0 \ar[r]\ar[d,two heads]& 0 \ar[r]\ar[d,"\cong"] & F(Z^iD^\bu) \ar[r]\ar[d,"f"] & FD^i \ar[r]\ar[d,"\cong"] & F(Z^{i+1}D^\bu)\ar[d,hookrightarrow]\\
					0 \ar[r] & 0\ar[r] & Z^iFD^\bu \ar[r] & FD^i \ar[r] &  FD^{i+1}
				\end{tikzcd}
			\]
			we see that $f$ is an isomorphism. Therefore by $(1b)$ it is $Z^iFD^\bu\in \Dd'$ for all $i$. It also follows that $FD^i\to Z^{i+1}FD^\bu$ is epic, so $FD^\bu$ is exact.
		\end{proof}

%% file: nVariableCase.tex
We now consider the case of multiple variables. Throughout the section we assume $n\geq 2$. The main example to keep in mind is the tensor-hom adjunction. First we will recall the definition and basic properties of $n$-variable adjunctions, see \cite{cgr2014} or \cite{Recktenwald2019} for a general treatment of $n$-variable adjunctions and \cite{Hovey02} for Quillen adjunctions in two variables. 

	We will give a new simplified criterion for an $n$-variable adjunction of abelian categories to induce a Quillen adjunction. In particular our criteria can be checked on the level of the underlying abelian category instead of chain complexes.
	
	\subsection{Multivariable Adjunctions}
		
		\begin{defi}\label{defi:n-var-adjunction}
			Let $\Aa_0,\ldots,\Aa_n$ be abelian categories. An \emph{adjunction of $n$ variables} consists of functors
			\[
				F:\Aa_1\times\ldots\times\Aa_n\to \Aa_0
			\]
			and
			\[
				G^j:\Aa_1^{op}\times \overset{\widehat{j}}{\ldots}\times \Aa_{n}^{op}\times \Aa_0\to \Aa_j
			\]
			for each $j$ (where the hat indicates that the entry $j$ is skipped) such that for each collection of elements $a_i\in \Aa_i$ there is an isomorphism
			\[
				\Mor_{\Aa_0}(F(a_1, \ldots, a_n),a_0)\cong \Mor_{\Aa_j}(a_j,G^j(a_1,\overset{\widehat{j}}{\ldots}, a_n,a_0))
			\]
			which is natural in each entry.
			We will say, that $F$ is an \emph{$n$-variable left adjoint with right adjoints $G^j$} to indicate an adjunction of $n$ variables.

		\end{defi}
		
		\begin{emp}\label{emp:restriction-of-adjoint}
			We briefly explain how an $n$-variable adjunction restricts to $(n-1)$-variable adjunctions.
			
			Let $F:\Aa_1\times\ldots\times\Aa_n\to \Aa_0$ be an $n$-variable left adjoint. If we fix an object $A_i\in \Aa_i$ for a fixed $i$ this clearly induces an $(n-1)$-variable left adjoint
			\[
				F(-,\ldots, A_i,\ldots, -):\Aa_1\times\skipind{i}\times\Aa_n\to \Aa_0.
			\]
			Its right adjoints are $G^j(-,\ldots, A_i,\ldots, -)$ for $j\neq i$.
			On the other hand if we consider the $n$-variable functor
			\[
				G^j:\Aa^{op}_1\times\skipind{j}\times \Aa_n^{op}\times\Aa_0\to \Aa_j
			\]
			and fix an object in $\Aa_i$ for $i\neq 0$ then it will again have $F$ as a left adjoint. However if we fix $A_0\in \Aa_0$, we have to take
			\[
				G^{j,op}:\Aa_1\times\skipind{j}\times\Aa_n\to \Aa_j^{op}
			\]
			and we see that it is a left $(n-1)$-variable adjoint as follows:
			\begin{align*}
				\Mor_{\Aa_j^{op}}(G^{j,op}(A_1,\skipind{j},A_n,A_0),A_j)&=\Mor_{\Aa_j}(A_j,G^{j}(A_1,\skipind{j},A_n,A_0))\\
					&\cong \Mor_{\Aa_0}(F(A_1,\ldots, A_n),A_0)\\
					&\cong \Mor_{\Aa_k}(A_k,G^k(A_1,\skipind{k},A_n,A_0))
			\end{align*}
			and we see that the right adjoints of $G^{j,op}(-,\skipind{j}, -,A_0)$ are $G^k(-,\skipind{k},-,A_0)$.
		\end{emp}

		\begin{defi}\label{def:functor-on-complexes}
			Let $F:\Aa_1\times\ldots\times \Aa_n\to \Aa_0$ be an $n$-variable left adjoint with right adjoints $G^j$. We define
			\[
			\Ch(F):\Ch(\Aa_1)\times\ldots\times \Ch(\Aa_n)\to \Ch(\Aa_0)
			\]
			as the sum total complex of the $n$-complex
			\[
			C^\alpha = F(A_1^{\alpha_1},\ldots, A_n^{\alpha_n}) \text{ for } \alpha \in \Z^n.
			\]
			Similarly we define
			\[
			\Ch(G^j):\Ch(\Aa_1)\times \skipind{j}\times \Ch(\Aa_n)\to \Ch(\Aa_0)
			\]
			as the product total complex of the $n$-complex
			\[
				D^\alpha = G^j(A_1^{-\alpha_1},\skipind{j},A_n^{-\alpha_n},A_0^{\alpha_0}) \text{ for } \alpha\in \Z^n.
			\]
		\end{defi}
		
		\begin{emp}
			We claim that $F$ and $G^j$ also induce functors on chain complexes. There are various technical hurdles to proving such a result. We believe that the clearest formulation can be achieved with the language of fibrations and opfibrations of \multi{categories}. We refrain from including the formalism here for brevity.
			For details, as well as a proof that the induced functors on chain complexes are an adjunction of $n$ variables, see \cite[Section 1.5]{Recktenwald2019}, in particular Theorem 1.5.11. 
		\end{emp}

	\subsection{Quillen adjunctions of  multiple variables}
	
		In \cite{Hovey99} Hovey goes over the general theory of how to generalize the notion of Quillen adjunctions to the case of two variables. In particular the example of the $\otimes \dashv \Shom$ adjunction for modules or sheaves has garnered much interest. The following is direct generalization of the construction used in the definition for Quillen adjunctions of two variables. 
		
		\begin{defi}
			Let $f:A\to B$ be a morphism in a category $\Cc$. We can view it as a functor $f:\Delta_1\to \Cc$ where $\Delta_1$ is the 1-simplex category $\set{0<1}$. Given a collection of morphisms $f_i:A_i\to B_i\in \Cc_i$, there is an induced functor $\Delta_1^n\to \Cc_1\times\ldots\times \Cc_n$. Let 
			\[
				F:\Cc_1\times\ldots\times\Cc_n\to \Cc_0,
			\]
			be a functor. Composition with $F$ induces $F(f_1,\ldots, f_n):\Delta_1^n\to \Cc_0$. If $\Cc_0$ has small colimits we can define the \emph{pushout product}
			\[
				\square_F (f_i) = \square_F(f_1,\ldots,f_n): \colim_{\Delta_1^n\backslash\set{(1,\ldots, 1)}} F(f_1,\ldots, f_n) \to \colim_{\Delta_1^n}F(f_1,\ldots, f_n).
			\]
			Note that $\Delta_1^n$ has the terminal object $(1,\ldots, 1)$, therefore the colimit on the right is simply $F(B_1,\ldots, B_n)$. 
			
			Similarly we define the \emph{pullback product}
			\[
				\blacksquare_F (f_i) = \blacksquare_F(f_1,\ldots,f_n): \lim_{\Delta_1^n} F(f_1,\ldots, f_n)\to \lim_{\Delta_1^n\backslash \set{(0,\ldots, 0)}} F(f_1,\ldots, f_n).
			\]
		\end{defi}
		\begin{example}
			Let $\Aa=\lmod{R}$ for a commutative ring and consider $\otimes:\Aa\times\Aa\to \Aa$. Then the pushout product of two morphisms $f:A\to B$ and $g:C\to D$ in $\Aa$ is
			\[
				A\otimes D\underset{A\otimes C}{\sqcup} B\otimes C\to B\otimes D.
			\]
		\end{example}

		The following properties of the pushout product are important for us. Clearly there are dual properties for the pullback product. 
		
		For an $n$-variable functor we can fix an object in one of the domain categories to obtain a functor in $(n-1)$ variables. We will want to argue by induction later, therefore it is useful to understand how the pushout product behaves for these restricted functors.
		
		\begin{lemma}\label{lem:pushout-product-and-restriction}
			Let $F:\Aa_1\times\ldots\times\Aa_n\to \Aa_0$ be a functor of abelian categories. Fix an object $X\in \Aa_n$. Then
			\[
				\square_{F(-,X)}(f_1,\ldots, f_{n-1})\cong \square_F(f_1,\ldots,f_{n-1},0\to X)
			\]
			is an isomorphism in $\Aa_0^{\Delta_1}$.
		\end{lemma}
			\begin{proof}
				We show that 
				\[
					\colim_{\Delta_1^n\backslash\set{(1,\ldots,1)}} F(f_1,\ldots, f_{n-1},0\to X)
				\]
				satisfies the universal property of 
				\[
					\colim_{\Delta_1^{n-1}\backslash\set{(1,\ldots,1)}} F(f_1,\ldots, f_{n-1},X).
				\]
				Let $Z$ be an object together with maps $F(X_1,\ldots, X_{n-1},X)\to Z$, where $X_i\in \set{A_i,B_i}$, but not all $X_i=B_i$, such that all resulting diagrams commute. In order to give a morphism
				\[
					\colim_{\Delta_1^n\backslash\set{(1,\ldots,1)}} F(f_1,\ldots, f_{n-1},0\to X)\to Z
				\]
				it is enough to give morphisms
				\[
					F(X_1,\ldots, X_{n-1},X)\to Z \text{ where }X_i\in \set{A_i,B_i} \text{ but not all }X_i=B_i
				\]
				since $F(X_1,\ldots, X_{n-1},0)=0$ for all possible collections $(X_i)$. This is exactly the given datum, hence the result follows from the uniqueness of colimits.
			\end{proof}
	
		\begin{lemma}\label{lem:coker-of-pushout-product}
			Let $F:\Aa_1\times\ldots\times \Aa_n\to \Aa_0$ be right exact in each variable and let $f_i:A_i\to B_i$ be a morphisms in $\Aa_i$ with cokernel $C_i$ for $i\in \set{1,\ldots, n}$. Then \setcounter{equation}{0}
			\begin{equation}
				F(B_1,\ldots, B_{n-1},A_n)\to \colim_{\Delta_1^n\backslash \set{(1,\ldots,1)}} F(f_1,\ldots, f_n)\to \colim_{\Delta_1^n\backslash \set{(1,\ldots,1)}} F(f_1,\ldots,f_{n-1},0\to C_n)\to 0
			\end{equation}
			and
			\begin{equation}\label{eq:lem:coker-of-pushout-product}
				\begin{tikzcd}
					\colim_{\Delta_1^n\backslash \set{(1,\ldots,1)}} F(f_1,\ldots, f_n)\ar[r,"{\square_Ff_i}"] & F(B_1,\ldots, B_n)\ar[r] & F(C_1,\ldots, C_n)\ar[r] &0
				\end{tikzcd}
			\end{equation}
			are exact. In particular $\coker(\square_Ff_i)\cong F(C_1,\ldots, C_n)$.
		\end{lemma}
			\begin{proof}
				We begin with the first sequence. In $\Aa_n^{\Delta_1}$ 
				\[
					\id_{A_n}\to f\to (0\to C_n)\to 0
				\]
				is exact, as is obvious from the following diagram
				\[
					\begin{tikzcd}
						A_n\ar[r]\ar[d] & A_n\ar[d]\\
						A_n\ar[r]\ar[d] & B_n\ar[d]\\
						0_n\ar[r]\ar[d] & C_n\ar[d]\\
						0_n\ar[r] & 0_n
					\end{tikzcd}
				\]
				Now note that
				\[
					\colim_{\Delta_1^n} F(f_1,\ldots, f_{n-1},\id_{A_n})=F(B_1,\ldots, B_{n-1},A_n)
				\]
				and that $F(B_1,\ldots, B_{n-1},A_n)$ appears in the diagram $F(f_1,\ldots,f_{n-1},\id_{A_n})$ (to be precise at position $(1,\ldots, 1,0)$) we see
				\[
					\colim_{\Delta_1^n\backslash \set{(1,\ldots,1)}} F(f_1,\ldots, f_{n-1},\id_{A_n})\cong F(B_1,\ldots, B_{n-1},A_n)
				\]
				We take the colimit over $\Delta_1^n\backslash \set{(1,\ldots, 1)}$ of the exact sequence of diagrams
				\[
					F(f_1,\ldots, f_{n-1},\id_{A_n})\to F(f_1,\ldots, f_{n})\to F(f_1,\ldots, f_{n-1},0\to C_n)\to 0 
				\]
				and the claim follows.
				
				The second statement is proved via induction. The case $n=1$ is trivial. Now consider the exact diagram, where we simply write $c$ for $\colim_{\Delta_1^n\backslash\set{(1,\ldots, 1)}}$
				\[
					\begin{tikzcd}
						F(B_1,\ldots, B_{n-1},A_n)\ar[d,equal]\ar[r] & cF(f_1,\ldots,f_n)\ar[d,"\square_Ff_i"] \ar[r]& cF(f_1,\ldots, f_{n-1},0\to{C_n}) \ar[r]\ar[d,"(*)"] & 0\\
						F(B_1,\ldots,B_{n-1},A_n)\ar[r] \ar[d]& F(B_1,\ldots, B_n)\ar[r] \ar[d]& F(B_1,\ldots, B_{n-1},C_n)\ar[r] \ar[d]& 0\\
						0 \ar[r] & \coker(\square_Ff_i)\ar[r] & \coker(*) \ar[r] & 0.
					\end{tikzcd}
				\]
				By \cref{lem:pushout-product-and-restriction} the morphism marked $(*)$ is $\square_{F(-,C_n)}(f_1,\ldots,f_{n-1})$. Also 
				\[
					F(-,C_n):\Aa_1\times\ldots\times \Aa_{n-1}\to \Aa_0
				\]
				is right exact in each variable, therefore we can apply the induction hypothesis and see that 
				\[
					\coker(\square_{F(-,C_n)}(f_1,\ldots, f_{n-1}))\cong F(C_1,\ldots, C_n)
				\]
				and by exactness
				\[
					\coker(\square_{F(-,C_n)}(f_1,\ldots, f_{n-1}))\cong \coker(\square_Ff_i).
				\]
			\end{proof}
		\begin{lemma}\label{lem:pushout-product-pushout-square}
			The following square is a pushout
			\[
				\begin{tikzcd}
					\colim_{\Delta_1^n\backslash \set{(1,\ldots,1)}}F(f_1,\ldots, f_{n-1},0\to A_n)\ar[r]\ar[d,"{\square_{F(-,A_n)}(f_1,\ldots, f_{n-1})}"] & \colim_{\Delta_1^n\backslash \set{(1,\ldots,1)}}F(f_1,\ldots, f_{n-1},0\to B_n)\ar[d]\\
					F(B_1,\ldots, B_{n-1},A_n)\ar[r] & \colim_{\Delta_1^n\backslash \set{(1,\ldots,1)}}F(f_1,\ldots,f_n)
				\end{tikzcd}.
			\]
		\end{lemma}
			\begin{proof}
				Let
				\[
					\begin{tikzcd}
						\colim_{\Delta_1^n\backslash \set{(1,\ldots,1)}}F(f_1,\ldots, f_{n-1},0\to A_n)\ar[r]\ar[d] & \colim_{\Delta_1^n\backslash \set{(1,\ldots,1)}}F(f_1,\ldots, f_{n-1},0\to B_n)\ar[d]\\
						F(B_1,\ldots, B_{n-1},A_n)\ar[r] & Z
					\end{tikzcd}.
				\]
				be a commutative square. 
				
				We have to give morphisms 
				\[
					F(X_1,\ldots, X_n)\to Z
				\]
				where $X_i\in \set{A_i,B_i}$ but not all $X_i=B_i$. Let $X_n=A_n$, then all possible collection 
				\[
					F(X_1,\ldots, X_{n-1},A_n)
				\] 
				where not all $X_i=B_i$ can be found in $F(f_1,\ldots, f_{n-1}, 0\to A_n)$. For $X_n=B_n$ they can be found in $F(f_1,\ldots, f_{n-1}, 0\to B_n)$. And finally $F(B_1,\ldots, B_{n-1},A_n)\to Z$ is explicitly given.
			\end{proof}
			
		\begin{lemma}\label{lem:p-product-adjunction}
			Let $F:\Aa_1\times\ldots\times\Aa_n\to \Aa_0$ be an $n$-variable left adjoint with right adjoints $G^i:\Aa_1^{op}\times\skipind{i}\times\Aa_n^{op}\times \Aa_0\to \Aa_i$. Fix an index $j>0$, an object $B_j$ of $\Aa_j$, and morphisms $f_i$ in $\Aa_i$ for $j\neq i>0$.

			Then for every object $A_0$ of $\Aa_0$

			\[
				\Hom_{}\left(\colim_{\scriptscriptstyle\Delta_1^n\backslash\set{(1,\ldots,1)}} F(f_1,\ldots,0\to B_j,\ldots, f_n),A_0\right)\cong \Hom_{}\left(B_j,\lim_{\scriptscriptstyle\Delta_1^n\backslash\set{(0,\ldots,0)}}G^j(f_1^{op},\skipind{j},f_n^{op},A_0\to 0)\right)
			\]
			is a natural isomorphism.
		\end{lemma}
			\begin{proof}
				Without loss of generality let $j=n$. By \cref{lem:pushout-product-and-restriction} 
				\[
					\colim_{\Delta_1^n\backslash \set{1,\ldots, 1}} F(f_1,\ldots,f_{n-1}, 0\to B_n)\cong \colim_{\Delta_1^{n-1}\backslash \set{1\ldots, 1}} F(f_1,\ldots, f_{n-1},B_n).
				\]
				Hence, by the universal property a morphism
				\[
					\colim_{\Delta_1^n\backslash\set{(1,\ldots,1)}} F(f_1,\ldots, f_n)\to A_0
				\]
				is a compatible collection of morphisms
				\[
					F(X_1,\ldots, X_{n-1},B_n)\to A_0
				\]
				where $X_i\in \set{A_i,B_i}$ not all $X_i=B_i$. By the adjunction such a morphism is the same as 
				\[
					B_n\to G^n(X_1,\ldots, X_{n-1},A_0),
				\]
				i.e.\@ a morphism
				\[
					B_n\to \lim_{\Delta_1^{n-1}\backslash \set{(0,\ldots, 0)}} G^n(f_1^{op},\ldots, f_{n-1}^{op}, A_0).
				\]
				The dual statement of \cref{lem:pushout-product-and-restriction} gives 
				\[
					\lim_{\Delta_1^{n}\backslash \set{(0,\ldots, 0)}} G^n(f_1^{op},\ldots, f_{n-1}^{op}, A_0\to 0),
				\]
				as desired.
			\end{proof}
		
		We can now use the pushout and pullback product to generalize \cref{defi:right-left-split} to multiple variables.
		
		\begin{defi}
			Let $F:\Aa_1\times \ldots\times \Aa_n\to \Aa_0$ be a functor. Let $\Ss_i\subset \Ob(\Aa_i)$ for $i=1,\ldots, n$ be classes of objects containing zero. We say $\Ss_1,\ldots, \Ss_n$ is \emph{$F$-right split} if for monomorphisms $f_i\in \Aa_i$ with cokernel in $\Ss_i$ the pushout product $\square_F (f_i)$ is a monomorphism.\\
			Let $G:\Aa_1\times\ldots,\times \Aa_n\to \Aa_0$ be a functor. Let $\Ss_i\subset \Ob(\Aa_i)$ for $i=1,\ldots, n$ be classes of objects. We say $\Ss_1,\ldots, \Ss_n$ is \emph{$G$-left split} if for epimorphisms $g_i\in \Aa_i$ with kernel in $\Ss_i$ the pullback product $\blacksquare_G (g_i)$ is an epimorphism.
		\end{defi}
		\begin{rem}
			Note that 
			\[
				\Ss_1,\ldots,\Ss_n \text{ is $F$-right split}
			\]
			is different from 
			\[
				\text{the class }\Ss_1\times\ldots\times \Ss_n \text{ is $F$-right split}.
			\]
			The second statement means that if $f_i:A_i\to B_i$ is a collection of monomorphisms in $\Aa_i$ such that the cokernel of $f_i$ is in $\Ss_i$, the sequence
			\[
				F(A_1,\ldots, A_n)\to F(B_1,\ldots,B_n)
			\]
			is a monomorphism.
			
			The following proposition shows that for any cotorsion pair $(\Dd,\Ee)$ it is always true that $\Dd,\Ee$ is $\Hom$-left split. However clearly $\Dd\times\Ee$ is not $\Hom$-left split. To see this let us assume that $\Aa$ has enough injectives. We then fix the cotorsion pair $(\Ob(\Aa),\text{injectives})$. Let $f_1:A\to B$ be any epimorphism (the condition that the kernel has to be in $\Aa$ is empty) and $f_2=\id_X:X\to X$ the identity. Note that $f_2$ is epic and has injective kernel. If $\Aa\times \text{injectives}$ were $\Hom$-left split
			\[
				\Hom(B,X)\to \Hom(A,X)
			\]
			would have to be an epimorphism, which is not true in general. 

		\end{rem}
	\begin{prop}\label{prop:d-e-is-hom-left-split}
		Let $\Aa$ be an abelian category with a cotorsion pair $(\Dd,\Ee)$. Then $\Dd,\Ee$ is $\Hom: \Aa^{op}\times \Aa\to \Ab$ left split.
	\end{prop}
		\begin{proof}
			Let $0\to A\to B\to D\to 0$ and $0\to E\to F\to G\to 0$ be short exact sequences in $\Aa$, where $D\in \Dd$ and $E\in \Ee$ (note that this makes $0\to D\to B\to A\to 0$ into a short exact sequence in $\Aa^{op}$). We have to show that
			\[
				\Hom(B,F)\to \Hom(B,G)\times_{\Hom(A,G)}\Hom(A,F)
			\]
			is an epimorphism.
			We write $(-,-):=\Hom(-,-)$. Consider the following diagram.
			\begin{equation}\label{eq:split-diagram}
				\begin{tikzcd}
					& \Ext^1(D,G)& \Ext^1(D,F)\ar[l]& \overbrace{\Ext^1(D,E)}^{=0}\ar[l]\\
					\Ext^1(A,E)& (A,G)\ar[l]\ar[u] & (A,F)\ar[l,"f"']\ar[u] & (A,E)\ar[l,"g"'] \ar[u]& \ar[l]0\\
					\Ext^1(B,E)\ar[u]& (B,G)\ar[u,"a"']\ar[l] & (B,F)\ar[l,"f'"']\ar[u,"b"'] & (B,E)\ar[l,"g'"']\ar[u,"c"'] & \ar[l]0\\
					\underbrace{\Ext^1(D,E)}_{=0}\ar[u] & (D,G)\ar[l]\ar[u,"a'"'] & (D,F)\ar[l,"f''"']\ar[u,"b'"'] & (D,E)\ar[l,"g''"']\ar[u,"c'"'] & \ar[l] 0\\
					& 0\ar[u] & 0 \ar[u]& 0 \ar[u]
				\end{tikzcd}
			\end{equation}
			
			Let $X\in(B,G)\times_{(A,G)} (A,F)$ be any element. The claim is proven by diagram chasing.
			\begin{enumerate}
				\item 
					Let $x\in (B,G)$ and $y\in (A,F)$ be the images of $X$ under the projections.
				\item 
					By exactness $f(y)\in \ker\left((A,F)\to\Ext^1(A,E)\right)$ .
				\item 
					Because $a(x)=f(y)$ and since $\Ext^1(D,E)=0$ we conclude $x\in \ker\left( (B,G)\to \Ext^1(B,E)\right)$.
				\item 
					By exactness there is $\tilde{x}\in (B,F)$ with $f'(\tilde{x})=x$. Similarly we find $\tilde{y}\in (B,F)$ with $b(\tilde{y})=y$.
				\item 
					Let $z=\tilde{x}-\tilde{y}$. Note that $af'(z)=0$ in $(A,G)$.
				\item 
					Again by exactness we find $\tilde{z}\in (D,G)$ with $a'(\tilde{z})=f'(z)$.
				\item
					By surjectivity there is $\bar{z}\in (D,F)$ with $f''(\bar{z})=\tilde{z}$.
				\item
					Let $z'=b'(\bar z)$. We claim that $z'+\tilde{y}$ is the desired preimage of $X$.
				\item
					We calculate 
					\[
						b(z'+\tilde{y})=b(b'(\bar{z}))+ b(\tilde{y})=y
					\]
					and
					\[
						f'(z'+\tilde{y}) = a'f''(\bar z) + f'(\tilde{y}) = a'(\tilde{z})+f'(\tilde{y}) = f'(z)+f'(\tilde{y})=f'(z+\tilde{y}) = f'(\tilde{x})=x.
					\]
			\end{enumerate}
		\end{proof}

		\begin{defi}\label{defi:sphere-and-disc}
			For an object $A$ in an abelian category $\Aa$ the \emph{sphere} and \emph{disc} complex are the complexes
			\begin{align*}
				S^n(A)&=\ldots\to 0\to \underbrace{A}_{\text{degree } n } \to 0 \to \ldots \\
				D^n(A)&=\ldots\to 0\to \underbrace{A}_{\text{degree } n } = \underbrace{A}_{\text{degree } n+1 } \to 0\to \ldots.
			\end{align*}
			For a complex $(A^\bu,d)$ we also define the \emph{cycles} $Z^nA^\bu=\ker(d^n:A^n\to A^{n+1})$ and \emph{boundaries} $B^nA^\bu=\im(d^{n-1}:A^{n-1}\to A^n)$.
		\end{defi}
	
		\begin{emp}\label{emp:two-missing-ind-notation}
			We introduce notation. Let
			\[
				G^j:\Aa_1^{op}\times\ldots\times \Aa_{j-1}\times\Aa_{j+1}\times\ldots\times \Aa_n^{op}\times \Aa_0\to \Aa_j
			\]
			be a functor and fix an index $k\neq j,0$. We want to restrict $G^j$ to a functor $\Aa_k^{op}\to \Aa_j$. So let fix objects $A_i\in \Aa_i$ for $i\neq j,k$. If $k<j$ we consider
			\[
				G^j(A_1,\ldots, A_{k-1},-,A_{k+1},\ldots, A_{j-1},A_{j+1},\ldots, A_n,A_0)
			\]
			and if $k>j$ we consider
			\[
				G^j(A_1,\ldots, A_{j-1},A_{j+1},\ldots, A_{k-1},-,A_{k+1},\ldots, A_n,A_0).
			\]
			In both cases we simply write $G^j(A_1,\overset{\widehat{j},k}{\ldots},A_n,A_0)$ for brevity.
		\end{emp}
		\begin{prop}\label{prop:n-split-duality}
			Let $F:\Aa_1\times\ldots\times\Aa_n\to \Aa_0$ be an $n$-variable left adjoint with right adjoints $G^j$ between abelian categories with cotorsion pairs $(\Dd_i,\Ee_i)$. Assume that every object of $\Aa_i$ is a quotient of an object in $\Dd_i$ as well as a subobject of an object in $\Ee_i$.
			
			Then the following collections of conditions are equivalent
			
			\vspace{ 1 em}
			\begin{minipage}[l]{0.45\textwidth}
				\begin{enumerate}
					\item[$(0a_k)$] $\Dd_k$ is $F(D_1,\ldots, \overset{k}{-},\ldots, D_n)$-right split for all $D_i\in \Dd_i$ ($i\neq k$) and all $k$.
					\item[$(0b)$] $F(\Dd_1,\ldots, \Dd_n)\subset \Dd_0$.
				\end{enumerate}
			\end{minipage}
			\begin{minipage}[c]{0.1\textwidth}
				
			\end{minipage}
			\begin{minipage}[r]{0.45\textwidth}
				\begin{enumerate}
					\item[$(ja_0)$] $\Ee_0$ is $G^j(D_1,\skipind{j},D_n,-)$-left split for all $D_i\in \Dd_i$.
					\item[$(ja_k)$] $\Dd_k$ is $G^j(D_1,\overset{\widehat{j},k}{\ldots},D_n,E)$-right split for all $D_i\in \Dd_i$ ($i\neq j,k$) and $E\in \Ee_0$.
					\item[$(jb)$\,] $G^j(\Dd_1,\skipind{j},\Dd_n,\Ee_0)\subset \Ee_j$.
				\end{enumerate}
			\end{minipage}
		\end{prop}
		Note that the conditions numbered $(0?)$ concern the functor $F$, while the conditions numbered $(j?)$ concern the functor $G^j$. Similarly $(?a_k)$ concern objects in $\Aa_k$. Therefore there is no condition $(ja_j)$.
			\begin{proof}
				The same argument used in the proof of \cref{prop:equivalent-conditions-1-variable} works. Indeed assume that $(0a_k)$ and $(0b)$ hold for all $k$. Then by fixing $D_i\in \Dd_i$ for $i\neq j$, we find an adjunction in one variable
				\[
					F(D_1,\empind{j},D_m)\dashv G^j(D_1,\skipind{j},D_n,-)
				\]
				and by \cref{prop:equivalent-conditions-1-variable} $(ja_0)$ and $(jb)$ hold for all $j$. It remains to show that $(ja_k)$ holds. Let
				\[
					0\to A_k\to B_k\to D_k\to 0
				\]
				be a short exact sequence with $D_k\in \Dd_k$, and fix $D_i\in \Dd_i$ for $i\neq j,k$ and $E\in \Ee_0$. Without loss of generality let $j<k$. We find
				\[
					\begin{tikzcd}[column sep = huge, row sep = small]
						0\ar[d]\\
						G^j(D_1,\skipind{j},D_{k-1},D_k,D_{k+1},\ldots, D_n,E)\ar[dd]\\
						\\
						G^j(D_1,\skipind{j},D_{k-1},B_k,D_{k+1},\ldots, D_n,E)\ar[dd]\\
						\\
						G^j(D_1,\skipind{j},D_{k-1},A_k,D_{k+1},\ldots, D_n,E)\ar[d] & D_j\ar[two heads, l]\ar[luu, dashed,"\exists?", end anchor = south east]\\
						0
					\end{tikzcd}
				\]
				where $D_j\to G^j(D_1,\skipind{j},D_{k-1},A_k,D_{k+1},\ldots, D_n,E)$ is a surjection with $D_j\in \Dd_j$. We have to show that this lifts to $G^j(D_1,\skipind{j},D_{k-1},B_k,D_{k+1},\ldots, D_n,E)$. Consider
				\[\adjustbox{max width=1.021\textwidth, center}{
					\begin{tikzcd}[column sep = scriptsize]
						\Hom(D_j,G^j(D_1,\skipind{j},D_{k-1},B_k,D_{k+1},\ldots, D_n,E))\ar[d]\ar[r,"\cong"] & \Hom(F(D_1,\ldots, D_{k-1}, B_k, D_{k+1},\ldots, D_n),E)\ar[d]\\
						\Hom(D_j,G^j(D_1,\skipind{j},D_{k-1},A_k,D_{k+1},\ldots, D_n,E))\ar[r,"\cong"] & \Hom(F(D_1,\ldots, D_{k-1}, A_k, D_{k+1},\ldots, D_n),E)\ar[d]\\
							& \Ext^1(F(D_1,\ldots, D_n),E)
					\end{tikzcd}}
				\]
				and note that by $(jb)$ it is $\Ext^1(F(D_1,\ldots, D_n),E)=0$, therefore the lift exists.\\
				The converse is left to the reader.
			\end{proof}

		\begin{prop}\label{prop:hovey-gen}
			Let 
			\[
				F:\Aa_1\times\ldots\times\Aa_n\to \Aa_0
			\]
			be an $n$-variable left adjoint between abelian categories. Let $(\Dd_i,\Ee_i)$ be a complete cotorsion pair on $\Aa_i$ for each $i$. Suppose the equivalent conditions of \cref{prop:n-split-duality} hold, e.g. for $F$ we assume
			\begin{enumerate}
				\item[$(0a_j)$] $\Dd_j$ is $F(D_1,\empind{j},D_n)$-right split for each $j$ and for all $D_i\in \Dd_i$ ($i\neq j$), and
				\item[$(0b)$] $F(\Dd_1,\ldots,\Dd_n)\subset \Dd_0$.
			\end{enumerate}
			Then $\Dd_1,\ldots,\Dd_n$ is $F$-right split.\\
			Additionally for short exact sequences
			\[
				\begin{tikzcd}
					0\ar[r] & A_i\ar[r,"f_i"]& B_i\ar[r] & \underbrace{D_i}_{\in \Dd_i}\ar[r] &0
				\end{tikzcd}\text{ in }\Aa_i
			\]
			the cokernel of $\square_F(f_1,\ldots, f_n)$ is isomorphic to $F(D_1,\ldots, D_n)$.
		\end{prop}
			\begin{proof}
				For each $i=1,\ldots, n$ let 
				\[
					\begin{tikzcd}
						0\ar[r] & A_i\ar[r,"f_i"]& B_i\ar[r] & \underbrace{D_i}_{\in \Dd_i}\ar[r] &0
					\end{tikzcd}
				\]
				be a short exact sequence.\\
				Via induction we will show that $\square_F (f_1,\ldots, f_n)$ is a monomorphism and its cokernel is isomorphic to $F(D_1,\ldots, D_n)$.
				
				The case $n=1$ is trivial. 
				
				Assume the result has been established for all $m<n$. Since $F$ is an $n$-variable left adjoint it is right exact in each variable. Therefore we can conclude $\coker(\square_F(f_i))\cong F(D_1,\ldots, D_n)$ by \cref{lem:coker-of-pushout-product}.
				
				It remains to show that $\square_F(f_1,\ldots, f_n)$ is a monomorphism. We write $0_X$ for the zero morphism $0\to X$ and $c$ for $\colim_{\Delta_1^n\backslash\set{(1,\ldots,1)}}$. By \cref{lem:pushout-product-pushout-square} 
				\[
					\begin{tikzcd}
						cF(f_1,\ldots, f_{n-1},0_{A_n})\ar[r]\ar[d] & cF(f_1,\ldots, f_{n-1},0_{B_n})\ar[d]\\
						F(B_1,\ldots, B_{n-1},A_n)\ar[r] & cF(f_1,\ldots, f_n)
					\end{tikzcd}
				\]
				is a pushout square. The following diagram
				\begin{equation}\label{eq:hovey-gen-diagram}
					\begin{tikzcd}[column sep = tiny, row sep = small]
						&& && && 0\ar[dd]\\
						\\
						&& cF(f_1,\ldots,f_{n-1},0_{A_n})\ar[rr]\ar[dd] &&  cF(f_1,\ldots,f_{n-1},0_{B_n})\ar[rr]\ar[dd]\ar[dl] &&  cF(f_1,\ldots,f_{n-1},0_{D_n})\ar[rr]\ar[dd,"(*)"] && 0\\
						&&& \square\ar[rd,dashed,"\square_F(f_i)"]
						\\
						&& F(B_1,\ldots, B_{n-1},A_n)\ar[rr]\ar[ur]\ar[dd] && 
						F(B_1,\ldots, B_n)\ar[rr] \ar[dd]&& 
						F(B_1,\ldots, B_{n-1},D_n)\ar[rr] \ar[dd]&& 
						0\\
						\\
						0\ar[rr] && F(D_1,\ldots, D_{n-1},A_n)\ar[rr] \ar[dd]&& F(D_1,\ldots, D_{n-1},B_n)\ar[rr]\ar[dd] && F(D_1,\ldots, D_n)\ar[rr] \ar[dd]&& 0\\
						\\
						&& 0 && 0 && 0 
					\end{tikzcd}
				\end{equation}
				has exact rows. To see that it also has exact columns we use \cref{lem:coker-of-pushout-product} and that the arrow labeled $(\ast)$ is monic by induction.
				
				By assumption the cotorsion pairs $(\Dd_i,\Ee_i)$ are complete and we find $E\in \Ee_0$ together with a monomorphism $\square\into E$. By the universal property every morphism $\square \to E$ corresponds uniquely to an element in
				\[
					\Hom(cF(f_1,\ldots,f_{n-1},0_{B_n}),E)\underset{\makebox*{$\Hom$}{\small$\Hom(cF(f_1,\ldots,f_{n-1},0_{A_n}),E)$}}{\times} \Hom(F(B_1,\ldots,B_{n-1},A_n),E).
				\]
				We claim that 
	
				\begin{equation}\label{eq:hovey-gen-claim}
					\begin{tikzcd}
						\Hom(F(B_1,\ldots, B_n),E)\ar[d]\\
						\Hom(cF(f_1,\ldots,f_{n-1},0_{B_n}),E)\underset{\makebox*{$\Hom$}{\small$\Hom(cF(f_1,\ldots,f_{n-1},0_{A_n}),E)$}}{\times} \Hom(F(B_1,\ldots,B_{n-1},A_n),E)
					\end{tikzcd}
				\end{equation}
				is surjective. That means there is a factorization
				\[
					\begin{tikzcd}[column sep = small]
						\square\ar[rr]\ar[rd] && E\\
						&F(B_1,\ldots, B_n)\ar[ru]
					\end{tikzcd}
				\]
				and since the composition is monic, so is $\square\to F(B_1,\ldots, B_n)$.
				Therefore it remains to show that \cref{eq:hovey-gen-claim} is indeed a surjection.
				
				We write $(-,-)=\Hom_{\Aa_0}(-,-)$ and apply $(-,E)$ to Diagram \labelcref{eq:hovey-gen-diagram}. The resulting diagram is Diagram \labelcref{eq:hovey-gen-diagram-hom} on \cpageref{eq:hovey-gen-diagram-hom}.
				
				Next we use that $F$ is an $n$-variable left adjoint. Denote by $G^j$ its right adjoints and use \cref{lem:p-product-adjunction} to see that this is isomorphic to the diagram obtained by applying $\Hom$ to the two short exact sequences
				\[
					0\to A_n\to B_n\to D_n
				\]
				and
				\[
					0\to G^n(D_1,\ldots, D_{n-1},E)\to G^n(B_1,\ldots, B_{n-1},E)\overset{(\dagger)}{\to}\lim_{\Delta_1^n\backslash \set{(0,\ldots,0)}} G^n(f_1^{op},\ldots, f_{n-1}^{op},E\to 0)\to 0.
				\]
				The second sequence is exact by induction and \cref{prop:n-split-duality}. We remind the reader of our comments on the fact that $G^n(-,\ldots, -,E)^{op}$ is a $(n-1)$ variable left adjoint (\cf{} \cref{emp:restriction-of-adjoint}). 
				
				The resulting diagram is Diagram \labelcref{eq:hovey-gen-diagram-adj} on \cpageref{eq:hovey-gen-diagram-adj}, where we write $l$ for $\lim_{\Delta_1^n\backslash\set{(0,\ldots,0)}}$.\\
				It then follows from \cref{prop:d-e-is-hom-left-split} that \cref{eq:hovey-gen-claim} is indeed a surjection.
				\newpage
				\begin{equation}\label{eq:hovey-gen-diagram-hom}
				\begin{aligned}
					\rotatebox{-90}{%
						\begin{tikzcd}[ampersand replacement = \&,row sep = huge, column sep = 1.75em]
							\& \&\& \underbrace{\Ext^1(F(D_1,\ldots, D_n),E)}_{=0}\\
							\& (cF(f_1,\ldots,f_{n-1}, 0_{A_n}),E) \& (cF(f_1,\ldots,f_{n-1},0_{B_n}),E) \ar[l]\& (cF(f_1,\ldots,f_{n-1},0_{D_n}),E) \ar[l]\ar[u]\& 0\ar[l]\\
							\& (F(B_1,\ldots,B_{n-1},A_n),E) \ar[u]\& (F(B_1,\ldots, B_n),E) \ar[l]\ar[u]\& (F(B_1,\ldots,B_{n-1},D_n),E) \ar[l]\ar[u]\& 0\ar[l]\\
							\underbrace{\Ext^1(F(D_1,\ldots, D_n),E)}_{=0} \& (F(D_1,\ldots,D_{n-1},A_n),E)\ar[u]\ar[l] \& (F(D_1,\ldots,D_{n-1}, B_n),E)\ar[l]\ar[u] \& (F(D_1,\ldots,D_n),E)\ar[l]\ar[u] \& 0\ar[l]\\
							\& 0 \ar[u] \& 0\ar[u] \& 0\ar[u]
						\end{tikzcd}
					}
				\end{aligned}
				\end{equation}
				\newpage
				
				\begin{equation}\label{eq:hovey-gen-diagram-adj}
				\begin{aligned}
					\rotatebox{-90}{%
						\begin{tikzcd}[ampersand replacement = \&,row sep = huge, column sep = 1.75em]
							\& \&\& 0\\
							\& (D_n,lG^n(f_1^{op},\ldots,f_{n-1}^{op},E\to 0)) \& (B_n,lG^n(f_1^{op},\ldots,f_{n-1}^{op},E\to 0)) \ar[l]\& (A_n,lG^n(f_1^{op},\ldots,f_{n-1}^{op},E\to 0)) \ar[l]\ar[u]\& 0\ar[l]\\
							\& (A_n,G^n(B_1,\ldots,B_{n-1},E)) \ar[u]\& (B_n,G^n(B_1,\ldots,B_{n-1},E))) \ar[l]\ar[u]\& (D_n,G^n(B_1,\ldots,B_{n-1},E)) \ar[l]\ar[u]\& 0\ar[l]\\
							0 \& (A_n,G^n(D_1,\ldots,D_{n-1},E)) \ar[u]\ar[l] \& (B_n,G^n(D_1,\ldots,D_{n-1},E))\ar[l]\ar[u] \& (D_n,G^n(D_1,\ldots,D_{n-1},E))\ar[l]\ar[u] \& 0\ar[l]\\
							\& 0 \ar[u] \& 0\ar[u] \& 0\ar[u]
						\end{tikzcd}
					}
				\end{aligned}
				\end{equation}
			\end{proof}
			
		This proposition can in particular be applied to the two cotorsion pairs on abelian model categories (\cf{} \cref{thm:cot-pairs-and-model-cats}). This will give criteria for adjunctions to be Quillen.
		
		The following definition is a variant of \cite[Definition 4.2.1]{Hovey99} and the theorem is a generalization of \cite[Theorem 7.2]{Hovey02}.
		\begin{defi}
			Let $\Mm_0,\ldots,\Mm_n$ be model categories with an $n$-variable left adjoint 
			\[
				F:\Mm_1\times\ldots\times\Mm_n\to \Mm_0.
			\]
			If for all collections of cofibrations $f_i$ in $\Mm_i$ for $i>0$
			\begin{enumerate}
				\item $\square_F (f_i)$ is a cofibration in $\Mm_0$, and
				\item $\square_F (f_i)$ is trivial when at least one of the $f_i$ is,
			\end{enumerate}
			then $F$ is an \emph{$n$-variable left Quillen functor}. 
		\end{defi}
		
		\begin{prop}\label{prop:equivalent-formulation-for-quillen-adjunction}
			Let $F:\Mm_1\times\ldots\times\Mm_n\to \Mm_0$ be an $n$-variable left adjoint with right adjoints $G^j$. For $i>0$ let $f_i$ be cofibrations in $\Mm_i$ and let $f_0$ be a fibration in $\Mm_0$. Then the following pairs of conditions are equivalent
			
			\vspace{ 1em }
			\noindent
			\hspace{-1em}
			\begin{minipage}[l]{0.5\textwidth}
				\begin{enumerate}
					\item[$(0a)$] $\square_F (f_i)_{i>0}$ is a cofibration in $\Mm_0$, and
					\item[$(0b)$] 
						it is trivial when at least one of the $f_i$ is.
				\end{enumerate}
			\end{minipage}
			\begin{minipage}[r]{0.5\textwidth}
				\begin{enumerate}
					\item[$(ja)$] $\blacksquare_{G^j} (f_i)_{i\neq j}$ is a fibration in $\Mm_j$, and
					\item[$(jb)$] 
					it is trivial when at least one of the $f_i$ is.
				\end{enumerate}
			\end{minipage}
			\vspace{0.8em}
			
			\noindent In one of these pairs of conditions hold, we will say that $F$ and the $G^j$ form an \emph{$n$-variable Quillen adjunction}. We call each of the $G^j$ an \emph{$n$-variable right Quillen functor}.
		\end{prop}
			\begin{proof}
				This is a simple exercise in adjunctions and the fact that (co)fibrations are characterized by lifting properties.

			\end{proof}

			\begin{thm}\label{thm:hovey-generalization}
				Let 
				\[
					F:\Mm_1\times\ldots\times\Mm_n\to \Mm_0
				\]
				be an $n$-variable left adjoint between abelian model categories with right adjoints $G^j$. We write $\Dd_i$ and $\Ee_i$ for the classes of cofibrant and fibrant objects in $\Mm_i$ respectively. Suppose the following conditions hold:
				\begin{enumerate}
					\item The family $\Dd_j$ is $F(D_1,\empind{j},D_n)$-right split for all $D_i\in \Dd_i$ ($i\neq j$) and all $j$.
					\item If $D_i\in \Dd_i$ for all $i>0$, then $F(D_1,\ldots, D_n)\in \Dd_0$.
					\item If additionally at least one of the $D_i$ is acyclic, then so is $F(D_1,\ldots,D_n)$.
				\end{enumerate}
				Then $F$ is an $n$-variable left Quillen functor.
			\end{thm}
				\begin{proof}
					By \cref{thm:cot-pairs-and-model-cats} there are two complete cotorsion pairs on $\Mm_i$, namely $(\Dd_i\cap \Ac_i,\Ee_i)$ and $(\Dd_i,\Ee_i\cap \Ac_i)$, where $\Ac_i$ denotes the class of trivial objects.
					
					Let $f_i$ be cofibrations in $\Mm_i$ for $i>0$, i.e.\@ there are short exact sequences
					\[
						\begin{tikzcd}
							0\ar[r] & A_i\ar[r,"f_i"]& B_i\ar[r] & \underbrace{D_i}_{\in \Dd_i}\ar[r] &0
						\end{tikzcd}\text{ in }\Mm_i
					\]
					By \cref{prop:hovey-gen} applied to the cotorsion pairs $(\Dd_i,\Ee_i\cap \Ac_i)$ the short sequence
					\[
						\begin{tikzcd}
							0\ar[r] & \colim_{\Delta_1^n\backslash\set{(1,\ldots,1)}} F(f_1,\ldots, f_n)\ar[r,"\square_F(f_i)"] & F(B_1,\ldots, B_n)\ar[r] & F(D_1,\ldots, D_n)\ar[r] & 0
						\end{tikzcd}
					\]
					is exact. By assumption $F(D_1,\ldots, D_n)$ is cofibrant, hence $\square_F(f_1,\ldots, f_n)$ is a cofibration in $\Mm_0$. 
					
					Now if at least one of the $f_i$ is trivial, then so is one of the $D_i$ and therefore $F(D_1,\ldots, D_n)$ is trivial. By \cref{lem:trivial-cofib} this implies that it is a trivial cofibration.

				\end{proof}
				The converse of the theorem is also true.
				\begin{prop}\label{prop:hovey-gen-converse}
					Let $F:\Mm_1\times\ldots\times\Mm_n\to \Mm_0$ be an $n$-variable left Quillen functor between abelian model categories. Then
					\begin{enumerate}
						\item the family $\Dd_j$ is $F(D_1,\empind{j},D_n)$-right split for all $D_i\in \Dd_i$ ($i\neq j$) and all $j$.
						\item if $D_i\in \Dd_i$ for all $i>0$, then $F(D_1,\ldots, D_n)\in \Dd_0$.
						\item If additionally at least one of the $D_i$ is acyclic, then so is $F(D_1,\ldots,D_n)$.
					\end{enumerate}
				\end{prop}
					\begin{proof}
						The second and third claim follow from \cref{lem:coker-of-pushout-product} and \cref{lem:trivial-cofib}.
						
						For the first claim let $j=n$ without loss of generality, and fix cofibrant objects $D_i$ of $\Mm_i$ for each $0<i<n$. Note that $0\to D_i$ is a cofibration. Let $f:A_n\to B_n$ be a monomorphism with cokernel in $\Dd_n$, i.e.\@ a cofibration. By assumption
						\[
							\square_F(0\to D_1,\ldots, 0\to D_{n-1}, f):\colim_{\Delta_1^n\backslash\set{(1,\ldots,1)}} F(0\to D_1,\ldots, 0\to D_{n-1}, f)\to F(D_1,\ldots, D_{n-1},B_n)
						\]
						is a monomorphism. It is easy to see
						\[
							\colim_{\Delta_1^n\backslash\set{(1,\ldots,1)}} F(0\to D_1,\ldots, 0\to D_{n-1}, f) \cong F(D_1,\ldots, D_{n-1},A_n)
						\]
						and 
						\[
							\square_F(0\to D_1,\ldots, 0\to D_{n-1}, f)= F(D_1,\ldots, D_{n-1},f).
						\]
					\end{proof}

		\subsection{A Simplified Criterion for Quillen Adjunctions in multiple variables}
			
				We will now use ideas similar to those in the one variable case to break the condition in \cref{thm:hovey-generalization} down into something that can be checked on the underlying abelian category instead of the category of chain complexes.
				
			\begin{emp}
				For the remainder of this section let $F:\Aa_1\times\ldots\times\Aa_n\to \Aa_0$ be an $n$-variable left adjoint with right adjoints $G^j$, where the $\Aa_i$ are Grothendieck abelian categories with cotorsion pairs $(\Dd_i,\Ee_i)$ which induce model structures on chain complexes.
			\end{emp}

			\begin{prop}\label{prop:n-variable-ext-adjunction}
				Assume the equivalent conditions of \cref{prop:n-split-duality}. Let $D_i^\bu \in \Ch(\Aa_i)$ be complexes such that $D_i^n\in \Dd_i$ for all $n$ and let $E^\bu \in \Ch(\Aa_0)$ be a complex such that $E^n\in \Ee_0$ for all $n$. Then
				\[
					\Ext^1_{\Ch(\Aa_0)}(\Ch(F)(D_1^\bu ,\ldots, D_n^\bu ),E^\bu )\cong \Ext^1_{\Ch(\Aa_j)}(D_j^\bu ,\Ch(G^j)(D_1^\bu ,\skipind{j},D_n^\bu ,E^\bu )).
				\]
			\end{prop}
				\begin{proof}
					Let 
					\[
						0\to E^\bu \to X^\bu\to \Ch(F)(D_1^\bu ,\ldots, D_n^\bu )\to 0
					\]
					be a short exact sequence of complexes. We apply $\Ch(G^j)(D_1^\bu ,\skipind j,D_n^\bu ,-)$ to find
					\[\adjustbox{max width=\textwidth, center}{
						\begin{tikzcd}[column sep = small]
						0\ar[r] &  \Ch(G^j)(D_1^\bu,\skipind{j},D_n^\bu,E^\bu) \ar[r] &  \Ch(G^j)(D_1^\bu,\skipind{j},D_n^\bu,X^\bu)\ar[r] &  \Ch(G^j)(D_1^\bu,\skipind{j},D_n^\bu,\Ch(F)(D_1^\bu,\ldots, D_n^\bu)) \ar[r,"?"] & 0\\
						0\ar[r] &  \Ch(G^j)(D_1^\bu,\skipind{j},D_n^\bu,E^\bu)\ar[r]\ar[u,equal] & \blacksquare \ar[r]\ar[u] & D_j^\bu\ar[r,"?"]\ar[u,"\eta"] & 0
						\end{tikzcd}}
					\]
					where the square denotes the pullback and $\eta$ the unit. We have to show exactness for the top row. Indeed then also the bottom row will be exact, and hence this construction gives an element of $\Ext^1_{\Ch(\Aa_j)}(D_j^\bu ,\Ch(G^j)(D_1^\bu ,\skipind{j},D_n^\bu ,E^\bu ))$. The dual construction will then give an inverse. \\
					Consider the sequence of \multi{complexes}
					\[\adjustbox{max width=\textwidth, center}{
						$0\to G^j(D_1^{\alpha_1},\skipind{j},D_n^{\alpha_n},E^{\alpha_0})\to G^j(D_1^{\alpha_1},\skipind{j},D_n^{\alpha_n},X^{\alpha_0})\to G^j(D_1^{\alpha_1},\skipind{j},D_n^{\alpha_n},\Ch(F)(D_1,\ldots, D_n)^{\alpha_0})\to 0$
						}
					\]
					with $\alpha=(\alpha_1,\skipind{j},\alpha_n,\alpha_0)\in \Z^n$.
					
					Since for all $\alpha$ the family $\Ee_0$ is left split for $G^j(D_1^{\alpha_1},\skipind{j},D_n^{\alpha_n},-)$ this is an exact sequence of \multi{complexes}. Also by assumption $G^j(D_1^{\alpha_1},\skipind{j},D_n^{\alpha_n},E^{\alpha_0})\in \Ee_0$, so we can apply \cref{coro:exactness-of-products} and after taking the product total complex it stays exact, as desired. 

				\end{proof}
			Recall the notation introduced in \cref{defi:sphere-and-disc}.
			\begin{lemma}\label{lem:tilde-criterion}
				Let $\Aa$ be an abelian category with a cotorsion pair $(\Dd,\Ee)$ such that every object of $\Aa$ is a quotient of an object of $\Dd$ and a subobject of an object of $\Ee$. 
				\begin{enumerate}
					\item 
						It is $D^\bu\in \tilde{\Dd}$ if and only if $\Ext^1_{\Ch(\Aa)} (D^\bu,S^n(E))=0$ for all $n\in \Z$ and $E\in \Ee$.
					\item
						It is $E^\bu\in \tilde{\Ee}$ if and only if $\Ext^1_{\Ch(\Aa)} (S^n(D),E^\bu)=0$ for all $n\in \Z$ and $D\in \Dd$.
				\end{enumerate}
			\end{lemma}

				\begin{proof}
					We only prove the first statement. 
					By \cref{lem:bounded-are-dg} $S^n(E)$ is a $\dg\tilde{\Ee}$-complex and by \cref{prop:enough-tilde-objects} $(\tilde{\Dd},\dg\tilde{\Ee})$ is a cotorsion pair. So for $D^\bu\in \tilde{\Dd}$ it is $\Ext^1_{\Ch(\Aa)} (D^\bu,S^n(E))=0$.
					
					Let $X^\bu$ be a complex such that $\Ext^1_{\Ch(\Aa)} (X^\bu,S^n(E))=0$ for all $n\in\Z$ and $E\in \Ee$. We first show that $X^\bu$ is exact. For any $E\in \Ee$ consider the exact sequence
					\[
						0\to S^{n+1}E\to D^nE\to S^nE\to 0
					\]
					and apply $\Hom(X^\bu,-)$ to it to find
					\[
						\Hom(X^\bu,D^nE)\to \Hom(X^\bu,S^nE)\to \underbrace{\Ext^1(X^\bu,S^{n+1}E)}_{=0}.
					\]
					Hence any morphism $X^\bu\to S^nE$ lifts to a morphism $X^\bu\to D^nE$. It is easy to see that a morphism $X^\bu\to S^nE$ is the same as a morphism $X^n/B^nX^\bu\to E$ in $\Aa$ and a morphism $X^\bu\to D^nE$ is the same as a morphism $X^{n+1}\to E$ in $\Aa$.
					
					Now let $X^n/B^nX^\bu\to E$ be a monomorphism for some $E\in \Ee$. Then we have just seen that this lifts to a commutative triangle
					\[
						\begin{tikzcd}
							X^n/B^nX^\bu\ar[r]\ar[d] & E\\
							X^{n+1}\ar[ru]
						\end{tikzcd}
					\]
					and we conclude that $X^n/B^nX^\bu\to X^{n+1}$ is a monomorphism, and therefore $B^nX^\bu=Z^nX^\bu$, i.e.\@ $X^\bu$ is exact at $n$.
					
					In particular it is $Z^{n+1}X\cong X^n/B^nX^\bu$. Then by \cite[Lemma 4.2]{Gillespie08} it is
					\[
						\Ext^1_\Aa(X^n/B^nX^\bu,E)\cong \Ext^1_{\Ch(\Aa)}(X^\bu,S^nE)=0
					\]
					for all $E\in \Ee$, hence $Z^{n+1}X^\bu\in \Dd$.
				\end{proof}
				
			We can now use these tools to prove the main result of this section. We will show that from the assumptions in \cref{prop:n-split-duality} one can infer the conditions of \cref{thm:hovey-generalization}. 
			
			\begin{thm}\label{thm:cot-main}
				We assume the equivalent conditions of \cref{prop:n-split-duality}, e.g.
				\begin{enumerate}
					\item[$(0a_j)$] $\Dd_j$ is $F(D_1,\ldots, \overset{j}{-},\ldots, D_n)$-right split for all $D_i\in \Dd_i$ ($i\neq j$) and all $j$.
					\item[$(0b)$] $F(\Dd_1,\ldots, \Dd_n)\subset \Dd_0$.
				\end{enumerate}
				Then $\Ch(F)$ is an $n$-variable left Quillen functor.
			\end{thm}
			
				\begin{proof}
					We check the conditions of \cref{thm:hovey-generalization}, which in this special case are:
					\begin{enumerate}
						\item
							The family $\dg\tilde{\Dd}_j$ is $\Ch(F)(D_1^\bu,\empind{j},D_n^\bu)$-right split for all $D_i^\bu\in \dg\tilde{\Dd}_i$ ($i\neq j$) for all $j$.
						\item
							If $D_i^\bu\in \dg\tilde{\Dd}_i$ for $i>0$ then $\Ch(F)(D_1^\bu,\ldots,D_n^\bu)\in \dg\tilde{\Dd}_0$.
						\item
							If additionally at least one of the $D_i^\bu$ is exact (i.e.\@ $D_i^\bu\in \tilde{\Dd}_i$), then $\Ch(F)(D_1^\bu,\ldots, D_n^\bu)$ is exact as well (i.e.\@ $\Ch(F)(D_1^\bu,\ldots, D_n^\bu)\in \tilde{\Dd}_0$).
					\end{enumerate}
					The first one can be shown directly. For the second and third we argue by induction. The case $n=1$ is \cref{thm:criterion-for-quillen-adjunction}.
					
					We begin with the first condition. Without loss of generality let $j=n$. Let $D_i^\bu\in \dg\tilde{\Dd}_i$ for $i=1,\ldots, n$, and let
					\[
						0\to A_n^\bu\to B_n^\bu\to D_n^\bu\to 0
					\]
					be a short exact sequence. Then by assumption for each $\alpha=(\alpha_1,\ldots, \alpha_n)\in \Z^n$
					\[
						0\to F(D_1^{\alpha_1},\ldots,D_{n-1}^{\alpha_{n-1}},A_n^{\alpha_n})
						\to F(D_1^{\alpha_1},\ldots,D_{n-1}^{\alpha_{n-1}},B_n^{\alpha_n})
						\to F(D_1^{\alpha_1},\ldots,D_n^{\alpha_n})
						\to 0
					\]
					is exact in $\Aa_0$ and $F(D_1^{\alpha_1},\ldots, D_n^{\alpha_n})\in \Dd_0$. By \cref{coro:exactness-of-products} 
					\[
						0\to \Ch(F)(D_1^\bu,\ldots, D_{n-1}^\bu, A_n^\bu)
						\to \Ch(F)(D_1^\bu,\ldots, D_{n-1}^\bu, B_n^\bu)
						\to \Ch(F)(D_1^\bu,\ldots, D_n^\bu)
						\to 0
					\]
					is exact as well, which concludes the argument for the first point.
					
					We now assume that the result has been established for all $m<n$.

					We start by showing the third condition. Without loss of generality let $D_n^\bu\in \tilde{\Dd}_n$.
					We have to show $\Ch(F)(D_1^\bu,\ldots, D_n^\bu)\in \tilde{\Dd}_0$, i.e. by \cref{lem:tilde-criterion} we need to prove
					\[
						\Ext^1(\Ch(F)(D_1^\bu,\ldots, D_n^\bu),S^mE)=0 \text{ for all } E\in \Ee_0\text{ and }m\in \Z.
					\]
					By \cref{coro:exactness-of-products} for each $k$ it is $\Ch(F)(D_1^\bu,\ldots, D_n^\bu)^k\in \Dd_0$, hence \cref{prop:n-variable-ext-adjunction} applies, so
					\[
						\Ext^1(\Ch(F)(D_1^\bu,\ldots, D_n^\bu),S^mE)\cong \Ext^1(D_1,\Ch(G^1)(D_2^\bu,\ldots, D_n^\bu,S^mE)).
					\]
					This reduces the statement to showing that $\Ch(G^1)(D_2^\bu,\ldots, D_n^\bu,S^mE)\in \tilde{\Ee}_1$. However \newline $\Ch(G^1)(D_2^\bu,\ldots, D_n^\bu,S^mE)$ is, up to a shift, just the functor induced by
					\[
						G^1(-,E):\Aa_2^{op}\times\ldots\times \Aa_n^{op}\to \Aa_0.
					\]
					So by \cref{prop:equivalent-formulation-for-quillen-adjunction} and the induction hypothesis $\Ch(G^1)(D_2^\bu,\ldots, D_n^\bu,S^0E)\in \tilde{\Ee}_1$ and after the shift $\Ch(G^1)(D_2^\bu,\ldots, D_n^\bu,S^mE)\in \tilde{\Ee}_1$. This concludes the proof of the third condition.

					Finally, for the second condition let $D_i^\bu\in \dg\tilde{\Dd}_i$ for all $i$ and $E^\bu\in \tilde{\Ee}_0$. It is
					\[
						\Ext^1(\Ch(F)(D_1^\bu,\ldots, D_n^\bu),E^\bu)\cong \Ext^1(D_1,\Ch(G^1)(D_2^\bu,\ldots, D_n^\bu,E^\bu)),
					\]
					hence it remains to show that $\Ch(G^1)(D_2^\bu,\ldots, D_{n}^\bu,E^\bu)\in \tilde{\Ee}_1$. The same argument as above can be used.
				\end{proof}

%% file: CotorsionPairsAndQuillenAdjunctions.bbl
\begin{thebibliography}{{Sta}18}

\bibitem[CGR14]{cgr2014}
E.~Cheng, N.~Gurski, and E.~Riehl.
\newblock Cyclic multicategories, multivariable adjunctions and mates.
\newblock {\em Journal of K-Theory}, 13(2):337 -- 396, 2014.

\bibitem[EJ00]{EnochsJenda00}
Edgar E. and Overtoun J.
\newblock {\em Relative Homological Algbra}, volume~30 of {\em De Gruyter
  exposition in mathematics}.
\newblock De Gruyter, 2000.

\bibitem[EO01]{EnochsOyonarte01}
Edgar Enochs and Luis Oyonarte.
\newblock Flat covers and cotorsion envelopes of sheaves.
\newblock {\em Proc. Amer. Math. Soc. 130 (2002), 1285-1292}, 2001.

\bibitem[Gil04]{Gillespie04}
James Gillespie.
\newblock The flat model structure on {C}h({R}).
\newblock {\em Transactions of the American Mathematical Society},
  356(08):3369--3391, 2004.

\bibitem[Gil06]{Gillespie06}
James Gillespie.
\newblock The flat model structure on complexes of sheaves.
\newblock {\em Transactions of the American Mathematical Society},
  358(7):2855--2874, 2006.

\bibitem[Gil08]{Gillespie08}
James Gillespie.
\newblock Cotorsion pairs and degreewise homological model structures.
\newblock {\em Homology, Homotopy and Applications}, 10(1):283--304, 2008.

\bibitem[Gil16]{Gillespie2016}
James Gillespie.
\newblock Gorenstein complexes and recollements from cotorsion pairs.
\newblock {\em Advances in Mathematics}, 291:859--911, 2016.

\bibitem[H{\"o}r17a]{1701.02152}
Fritz H{\"o}rmann.
\newblock Derivator six functor formalisms --- definition and construction i,
  2017.
\newblock preprint, \url{https://arxiv.org/abs/1701.02152}.

\bibitem[H{\"o}r17b]{Hoermann2017}
Fritz H{\"o}rmann.
\newblock Fibred multiderivators and (co)homological descent.
\newblock {\em Theory and Applications of Categories}, 32(38):1258--1362, 2017.

\bibitem[H{\"o}r18]{Hormann2018}
Fritz H{\"o}rmann.
\newblock Six-functor-formalisms and fibered multiderivators.
\newblock {\em Selecta Mathematica}, 24(4):2841--2925, 2018.

\bibitem[H{\"o}r19]{1902.03625}
Fritz H{\"o}rmann.
\newblock Derivator six-functor-formalisms - construction ii, 2019.
\newblock preprint, \url{https://arxiv.org/abs/1902.03625}.

\bibitem[Hov99]{Hovey99}
Mark Hovey.
\newblock {\em Model Categories}, volume~63 of {\em Mathematical Surveys and
  Monography}.
\newblock American Mathematical Society, 1999.

\bibitem[Hov02]{Hovey02}
M.~Hovey.
\newblock Cotorsion pairs, model category structures, and representation
  theory.
\newblock {\em Mathematische Zeitschrift}, 241(3):553--592, 2002.

\bibitem[Hov07]{Hovey07}
Mark Hovey.
\newblock {\em Interactions between homotopy theory and algebra}, chapter
  Cotorsion pairs and model categories, pages 277 -- 296.
\newblock Contemporary Mathematics. Amer. Math. Soc., Providence, RI, 2007.

\bibitem[Rec19]{Recktenwald2019}
Rene Recktenwald.
\newblock {\em Construction of six functor formalisms}.
\newblock PhD thesis, Albert-Ludwigs-University Freiburg, 2019.

\bibitem[{Sta}18]{stacks-project}
The {Stacks Project Authors}.
\newblock \textit{Stacks Project}.
\newblock \url{https://stacks.math.columbia.edu}, 2018.

\end{thebibliography}
